\documentclass{amsart}
\usepackage{amssymb}
\usepackage{amsfonts,color,wrapfig,epsfig,epic,amsthm,amsmath}
\usepackage{amsaddr}

\def\R{{\mathbb R}}

\def\<{\langle}
\def\>{\rangle}
\def\P{\mathbb P}

\def\E{\mathbb E}

\def\0{\underline 0}

\def\1{\underline 1}

\newcommand{\bel}{\begin{equation}\label}
\newcommand{\ee}{\end{equation}}
\newtheorem{theorem}{Theorem}[section]
\newtheorem{proposition}[theorem]{Proposition}
\newtheorem{corollary}[theorem]{Corollary}
\newtheorem{lemma}[theorem]{Lemma}
\newtheorem{remark}{Remark}[section]
\theoremstyle{definition}

\begin{document}
\title[]{Order statistics from overlapping samples: \\
bivariate densities and regression properties }
\author{F. L\'opez-Bl\'azquez$^1$, Nan-Cheng Su$^2$ and Jacek Weso\l owski$^3$}
\address[]{\footnotesize \hspace{-0.7cm}$^1$Departamento Estadistica e Investigati\'on Operativa, Universidad de Sevilla, Spain\newline
                         \hspace*{1.4cm}$^2$Department of Statistics, National Taipei University, Taiwan\newline
                         $^3$Faculty of Mathematics and Information Science, Warsaw University of Technology, Poland
                            }
\email[Corresponding author]{$^2$Corresponding author; sunanchen@gmail.com; Tel:886-2-86741111 ext
66778}
\date{\today }

\begin{abstract} In this paper we are interested in the joint distribution of two order statistics from overlapping samples. We give an explicit formula for the distribution of such a pair of random variables under the assumption that the parent distribution is absolutely continuous (with respect to the Lebesgue measure on the real line). The distribution is identified through the form of the density with respect to a measure which is a sum of the bivariate Lebesgue measure on $\R^2$ and the univariate Lebesgue measure on the diagonal $\{(x,x):\,x\in\R\}$.
	
We are also interested in the question to what extent conditional expectation of one of such order statistic given another determines the parent distribution. In particular, we provide a new characterization by linearity of regression of an order statistic from the extended sample given the one from the original sample, special case of which solves a problem explicitly stated in the literature. It appears that to describe the correct parent distribution  it is convenient to use quantile density functions. In several other cases of regressions of order statistics we provide new results regarding uniqueness of the distribution in the sample. Nevertheless the general question of identifiability of the parent distribution by regression of order statistics from overlapping samples remains open.

\bigskip


\noindent MSC 2010\textit{\ Subject Classifications}: Primary: 60E05;
secondary: 62E10
\end{abstract}

\maketitle


\section{Introduction}\label{intro}

\label{sec.1}
Properties of order statistics (os's) $X_{1:n} \leq X_{2:n} \leq \cdots \leq X_{n:n}$  based on the sample $\{X_1,X_2,\cdots,X_n\}$ of independent and identically distributed (iid) random variables with absolutely continuous distribution are widely known, see e.g. the monographs David and Nagaraja (2003) or Arnold, Balakrishnan and Nagaraja (2008) for excellent reviews. Much less is known for os's which arise from different samples which have common elements. There are two special cases which until now have been studied in the literature: (1) moving os's, when the subsequent  samples are of the same size and have the same size of the overlap - see e.g. Inagaki (1980), David and Rogers (1983), Ishida and Kvedaras (2015) or Balakrishnan and Tan (2016); (let us mention that moving samples have a long history in quality control and time series analysis - in particular,
the moving median is a simple robust estimator of location and the moving range is a current measure of dispersion complementing the moving average); (2) special cases of os's from the original and extended sample which except the original sample contains a number of additional observations - see e.g. Siddiqui (1970), Tryfos and Blackmore (1985), Ahsanullah and Nevzorov (2000) or L\'{o}pez-Bl\'azquez and Salamanca-Mi\~no (2014).

In the latter paper the authors introduced a reference measure $\nu$ with respect to which the joint distribution of $(X_{n-k+1:n},\,X_{n-k+2:n+1})$ (which they were interested in) has a density. This measure, $\nu $,  defined by
\begin{equation}\label{mesnu}
\nu (B)=\mu _{2}(B)+\mu _{1}(\pi (B)),\qquad B\in \mathcal{B}({\mathbb{R}}^{2}),
\end{equation}
where  $\pi (B)=\{x\in {\mathbb{R}}:\,(x,x)\in B\}$ and $\mu _{i}$ is the
Lebesgue measure in ${\mathbb{R}}^{i}$, $i=1,2$, will be of special interest for us here since $\nu$ will serve as the reference measure for bivariate densities of os's from overlapping samples.


In this paper we consider iid random variables $X_1,X_2,\ldots$ with  the cumulative distribution function (cdf)  denoted by $F$, its tail denoted by $\bar{F}:=1-F$ and the density with respect to $\mu_1$ denoted by $f$.

Let $\emptyset\ne A\subset \{1,2,\ldots\}$ be such that $n_A:=|A|<\infty$, where $|A|$ denotes the number of elements in $A$. By $X_{i:A}$ denote the $i$th os from the sample $\{X_k,\,k\in A\}$, $i=1,\ldots,n_A$.
In  case $A=\{1,\ldots,n\}$ we have $X_{i:A}=X_{i:n}$, $i=1,\ldots,n$. Consider additionally $\emptyset\ne B\subset \{1,2,\ldots\}$ such that $n_B:=|B|<\infty$. Our aim is to study the joint distribution of $(X_{i:A},\,X_{j:B})$, $i=1,\ldots,n_A$, $j=1,\ldots,n_B$. Of course, when $A\cap B=\emptyset$ the samples $\{X_k,\,k\in A\}$ and $\{X_k,\,k\in B\}$ are independent and the joint distribution of $(X_{i:A},\,X_{j:B})$ is just a product of marginal distributions of $X_{i:A}$ and $X_{j:B}$. We will only consider the case when $A\cap B\neq \emptyset$. Due to the permutation invariance of the distribution of $(X_1,X_2,\ldots)$ it suffices to take $A=\{1,\ldots,m\}$ and $B=\{r+1,r+2,\ldots,r+n\}$ with $r<m\le n$. Then we denote $X_{j:n}^{(r)}:=X_{j:B}$. In Section \ref{jointSec} we will study the joint density (with respect to the reference measure $\nu$) of the pair $(X_{i:m},\,X_{j:n}^{(r)})$. The case $r=0$ is technically much simpler but the main idea of the approach is the same as in the general case. Therefore we first derive the joint distribution of $(X_{i:m},\,X_{j:n})$ in Subsection \ref{simply} while the general case of an arbitrary $r\ge 0$ is considered in Subsection \ref{genfor} (with some technicalities moved to Appendix).

In Section \ref{Sec3} we are  interested in regressions $\E(X_{i:m}|X_{j:n}^{(r)})$, $\E(X_{j:n}^{(r)}|X_{i:m})$ and related characterizations or identifiability questions. The main tools are representations of these regressions in terms of combinations of $\E(X_{k:n+r}|X_{\ell:n+r})$, $k,l\in\{1,\ldots,n+r\}$. Since for $r>0$ such representations are rather complex,  our considerations in this case will be restricted to the simplest cases of regressions of $X_{1:2}$, $X_{2:2}$ given $X_{1:2}^{(1)}$ or given $X_{2:2}^{(1)}$. They are studied in Subsection \ref{mnr}.

The case of $r=0$ is much more tractable, though since $m<n$ the analysis of each of two dual regressions  $\E(X_{i:m}|X_{j:n})$ and $\E(X_{j:n}|X_{i:m})$ is quite different. In  particular, Do\l egowski and Weso\l owski (2015) (DW in the sequel) proved that
\begin{equation}\label{expe3}
\mathbb{P}(X_{i:m}=X_{k:n})=
\tfrac{\binom{k-1}{i-1}\binom{n-k}{m-i}}{\binom{n}{m}}I_{\{i,\ldots ,n-m+i\}}(k),
\end{equation}
and, consequently, obtained the following representation
\begin{equation}
\mathbb{E}(X_{i:m}|X_{j:n})=\sum_{k=i}^{n-m+i}\,\tfrac{\binom{k-1}{i-1}
\binom{n-k}{m-i}}{\binom{n}{m}}\,\mathbb{E}(X_{k:n}|X_{j:n}).
\label{expe1}
\end{equation}
Here and everywhere below equations involving conditional expectations are understood in the $\P$-almost sure sense.

It was proved in DW with the help of \eqref{expe1}  that the condition
\begin{equation}
\mathbb{E}(X_{i:m}|X_{j:n})=aX_{j:n}+b, \label{eq03}
\end{equation}
characterizes the parent distributions (exponential, Pareto and power) when $j\leq i$ and $j\geq n-m+i$. The case of $i=j$ had been considered earlier in Ahsanullah and Nevzorov (2000) even for an arbitrary shape of the regression function. The characterization through condition \eqref{eq03} given in DW is a direct generalization of characterizations by linearity of $\E(X_{i:n}|X_{j:n})$. Analysis of such problems has a long history - see e.g. references in DW, in particular, Ferguson (2002). In this case the complete answer was given in Dembi\'{n}ska and Weso\l owski (1998) through an approach based on the integrated Cauchy functional equation (see also L\'opez-Bl\'azquez and Moreno-Rebollo (1997) who used instead differential equations).  Actually, the question of determination of the parent distribution by the (non-linear) form  of regression $\E(X_{i:n}|X_{j:n})$ for non-adjacent $i$ and $j$ has not been completely resolved until now - see e.g. Bieniek and Maciag (2018).

In Subsection \ref{mn} we investigate characterizations by linearity of regression \eqref{eq03} in the remaining unsloved  cases, i.e. when $i<j<n-m+i$. In particular, we solve the easiest non-trivial open problem explicitly formulated in DW. The dual case of regressions of  an os from the extended sample given an os from the original sample, i.e. $\E(X_{j:n}|X_{i:m})$ is considered in Subsection \ref{nm}. The main results in this subsection identify several new situations in which the shape of the regression function determines uniquely the parent distribution.    Finally, some conclusions are discussed in Section \ref{sec.6}.

\section{Bivariate distribution of os's from overlapping samples}\label{jointSec}
In this section we will derive joint distribution of the pair $(X_{i:m},\,X_{j:n}^{(r)})$. This will be given through the density $f_{X_{i:m},\,X_{j:n}^{(r)}}$ with respect to the measure $\nu$ introduced in Section \ref{intro}. This density will be expressed as a linear combination of densities of pairs of os's $(X_{k:n+r},\,X_{\ell:n+r})$, $1\le k,\ell\le n+r$. The general formula is quite complicated technically as can be seen in Subsection \ref{genfor}, however the basic ideas are the same as in the simple case of $r=0$ which, as a warm up, is considered first in Subsection \ref{simply}.

\subsection{Original sample and its extension - the case of $r=0$} \label{simply}
Let $\mathbb{R}^n_{\ne}=\{\underline{x}\in \mathbb{R}^n:\; x_i\ne x_j \text{ for } i\ne j\}$ and $\mathbb{R}^n_{\uparrow}=\{\underline{x}\in \mathbb{R}^n:\; x_1<\cdots <x_n\}$. A vector with increasingly sorted components of  $\underline{x}=(x_1,\ldots, x_n)\in \mathbb{R}^n_{\ne}$ will be denoted by $\text{sort}_n(\underline{x}):=(x_{1:n},\ldots, x_{n:n})\in \mathbb{R}^n_{\uparrow}$ and $\sigma_n (\underline{x})=\tau\in \mathcal{S}_n$ (set of permutations of $\{1,\ldots,n\}$) defined by $\tau (i)=j$ if $x_{i:n}=x_{j}$. The correspondence $\underline{x}\in\mathbb{R}^n_{\ne}\leftrightarrow (\text{sort}_n(\underline{x}), \sigma_n (\underline{x}))\in \mathbb{R}^n_{\uparrow}\times \mathcal{S}_n$ is bijective.

For $\underline{x}\in\R^n_{\ne}$ denote $\underline{x}^{(m)}:=(x_1,\ldots,x_m)\in \mathbb{R}^m_{\ne}$, $m=1,\ldots,n$. Then  $\text{sort}_m(\underline{x}^{(m)})=(x_{1:m},\ldots, x_{m:m})$ and $\sigma_m(\underline{x}^{(m)})$, $m=1,\ldots,n$, are sequences of increasing lengths that keep track of the sorting up to the sequential observation of the $m$-th component of $\underline{x}$. For instance, if $\underline{x}=(2.3,\,1.7,\, 3.4,\, 2.5,\, 1.2)$ then

$$\begin{array}{lll}
& \text{sort}_1(\underline{x}^{(1)})=(2.3), & \sigma_1(\underline{x}^{(1)})=(1)\\
& \text{sort}_2(\underline{x}^{(2)})=(1.7,\,2.3), & \sigma_2(\underline{x}^{(2)})=(21)\\
& \text{sort}_3(\underline{x}^{(3)})=(1.7,\,2.3,\,3.4), & \sigma_3(\underline{x}^{(3)})=(213)\\
& \text{sort}_4(\underline{x}^{(4)})=(1.7,\,2.3,\, 2.5,\,3.4), & \sigma_4(\underline{x}^{(4)})=(2143)\\
& \text{sort}_5(\underline{x}^{(5)})=(1.2,\,1.7,\,2.3,\, 2.5,\,3.4), & \sigma_5(\underline{x}^{(5)})=(52143).\\
\end{array}$$

Observe that for $m<n$,  $\sigma_m(\underline{x}^{(m)})$ is obtained from  $\sigma_n(\underline{x})$ by deletion of $m+1,\, \ldots, n$.

Given a permutation $\tau \in \mathcal{S}_n$, let us denote by $\tau^{(m)}\in \mathcal{S}_m$ the permutation obtained from $\tau$ by deletion of the elements $m+1,\ldots, n$. For  (fixed) values $i,\, k,\, m,\, n$ such that $1\le i\le m $, $1\le k \le n$ and  $m\le n$, let us define
\[
A_{i:m;\, k:n}=\{\tau\in \mathcal{S}_n:\, \tau(k)=\tau^{(m)}(i)\}\in\mathcal{S}_n.
\]
Note that for any $\underline{x}\in \mathbb{R}^n_{\ne}$
\begin{equation}\label{180314-1221}
 \sigma_n(\underline{x})\in A_{i:m;\, k:n}\quad \Leftrightarrow \quad x_{i:m}=x_{k:n}.
\end{equation}
 For instance, in the previous example, $\sigma_5(\underline{x})=(52143)\in A_{2:4;\, 3:5}$ because $x_{2:4}=x_{3:5}=2.3$.

Since $(X_1,\ldots,X_n)$ has absolutely continuous distribution  $\underline{X}\in \mathbb{R}^n_{\ne}$ $\P$-a.s. Therefore, $\text{sort}_n(\underline{X})$ and $\sigma_n (\underline{X})$ are well defined $\P$-a.s. In particular, $\text{sort}_n(\underline{X})=(X_{1:n}, \ldots, X_{n:n})$ are the os's from the sample of size $n$.

\begin{lemma}\label{140407-1920}
	Random elements $\text{sort}_n(\underline{X})$ and $\sigma_n(\underline{X})$ are independent.
\end{lemma}
The result follows immediately from the fact that the distribution of $(X_1,\ldots,X_n)$ is invariant under permutation and that ties appear with probability zero.

For $n\ge 1$, and $k\neq j$ with $1\le j,\, k\le n$, it is well known that $(X_{k:n},\, X_{j:n})$ has a density with respect to $\mu_2$. This density, denoted here  by $f_{k,j:n}$, see, e.g. David and Nagaraja (2003), p.12 for the explicit expression in terms of $F$, $\bar{F}$ and $f$, satisfies
\bel{180319-1147}
\P (X_{k:n}\le x,\, X_{j:n}\le y)
=\iint\limits_{(-\infty,x]\times (-\infty,y]} f_{k,j:n}(s,t)\,d\mu_2(s,t)
=
\iint\limits_{(-\infty,x]\times (-\infty,y]} f_{k,j:n}(s,t)\,d\nu(s,t),
\ee
where for the  last equality to hold we chose a version of the density  $f_{k,j:n}$ satisfying $f_{k,j:n}(s,s)=0$, $s\in \mathbb{R}$. We also denote the density of $X_{j:n}$ by $f_{j:n}$ for more simplification.

If $k=j$, the random vector $(X_{k:n},\, X_{j:n})$ assumes values on the diagonal of $\R^2$ so that it does not have a density with respect to $\mu_2$, but it has a density with respect to $\nu$ of the form  $f_{j,j:n}(s,t)=f_{j:n}(s)\delta_{s,t}$ (with $\delta_{s,t}$ the Kronecker's delta). Indeed, we have
\begin{align}
&
\P (X_{j:n}\le x,\, X_{j:n}\le y)
=
\P (X_{j:n}\le \min(x,y))
=
\int_{-\infty}^{\min(x,y)} f_{j:n}(s)\, d\mu_1(s)
\nonumber
\\
&
=
\iint\limits_{(-\infty,x]\times (-\infty,y]} f_{j:n}(s)
\delta_{s,t}\,d\nu(s,t)
=
\iint\limits_{(-\infty,x]\times (-\infty,y]} f_{j,j:n}(s,t)\,d\nu(s,t).
\label{180319-1148}
\end{align}

\begin{theorem}
For integers $1\le i\le m $, $1\le j \le n$, $m\le n$, the random vector $(X_{i:m},\, X_{j:n})$ has an absolutely continuous distribution with respect to $\nu$ and the density function is of the form
\begin{equation}
\label{140410-1856}
f_{X_{i:m},\,X_{j:n}}(x,y)=
\sum_{k=i}^{i+m-n}
\frac{\binom{k-1}{i-1}\binom{m-k}{n-i}}{\binom{m}{n}} f_{k,j:n}(x,y).
\end{equation}

\end{theorem}

\begin{proof}
A consequence of \eqref{180314-1221} and  \eqref{expe3} is
\begin{equation}\label{180316-1155}
\P(\sigma_{n}(\underline{X})\in A_{i:m;k:n})
=
\P(X_{i:m}=X_{k:n})=
\tfrac{\binom{k-1}{i-1}\binom{n-k}{m-i}}{\binom{n
}{m}}I_{\{i,\ldots ,n-m+i\}}(k).
\end{equation}

Using Lem. \ref{140407-1920}, \eqref{180316-1155} and expressions \eqref{180319-1147} and \eqref{180319-1148},  we get
\begin{align*}
&
\P (X_{i:m}\le x,\, X_{j:n}\le y)=
\sum_{k=i}^{i+n-m} \P (X_{i:m}\le x,\, X_{j:n}\le y, \, X_{i:m}=X_{k:n})\nonumber
\\
&
=
\sum_{k=i}^{i+n-m} \P (X_{k:n}\le x,\, X_{j:n}\le y, \, \sigma_n (\underline{X})\in A_{i:m; \,k:n})\nonumber
\\
&
=
\sum_{k=i}^{i+n-m} \P (X_{k:n}\le x,\, X_{j:n}\le y)  \P( \sigma_n(\underline{X})\in A_{i:m;\, k:n})\nonumber
\\
&
=
\sum_{k=i}^{i+n-m}
\frac{
\binom{k-1}{i-1}
\binom{m-k}{n-i}
}
{
\binom{m}{n}
}
\P (X_{k:n}\le x,\, X_{j:n}\le y)
\\
&
=
\iint\limits_{(-\infty,x]\times (-\infty,y]}
\sum_{k=i}^{i+m-n}
\frac{
\binom{k-1}{i-1}
\binom{m-k}{n-i}
}
{\binom{m}{n}
} f_{k,j:n}(s,t)\, d\nu (s,\, t),
\end{align*}
which proves the assertion.

\hfill\end{proof}

Note that for $j\notin \{i,\ldots,i+n-m\}$ the distribution  $(X_{i:m},\, X_{j:n})$ is absolutely continuous with respect to the bivariate Lebesgue measure, $\mu_2$. On the contrary, for $j\in \{i,\ldots,i+n-m\}$, it has a singular part, so that there is no density function with respect to $\mu_2$. The advantage of the measure $\nu$ introduced in Section 1 is that the joint distribution of $(X_{i:m},\, X_{j:n})$ is absolutely continuous  with respect to $\nu$ in any case.

Formula \eqref{140410-1856} implies that conditional
distribution $\P_{X_{i:m}|X_{j:n}=y}$ has the density, with respect to the
measure $\nu _{y}$ defined by $\nu _{y}(B)=\mu_{1}(B)+\delta _{B}(y)$, $
B\in \mathcal{B}({\mathbb{R}})$, which reads
\begin{equation}
f_{X_{i:m}|X_{j:n}=y}(x)=\sum_{k=i}^{i+n-m}\,\tfrac{\binom{k-1}{i-1}\binom{n-k}{m-i}}{\binom{n}{m}}\,f_{X_{k:n}|X_{j:n}=y}(x)  \label{code1}
\end{equation}
where $f_{X_{j:n}|X_{j:n}=y}(x)=I_{\{y\}}(x)$. Consequently,  the formula for the conditional expectation of $X_{i:m}$ given $X_{j:n}$ as given in \eqref{expe1} follows.

\subsection{Overlapping samples - the general case of $r\ge 0$}\label{genfor}
In order to derive the formula for density of $(X_{i:m},\,X_{j:n}^{(r)})$ in the general case, $r\ge 0$, we need first to do a little bit of combinatorics of permutations, which will allow us to find the probabilities
$$
\P(X_{i,m}=X_{k:n+r},\,X_{j:n}^{(r)}=X_{\ell:n+r}),\qquad k,\,\ell\in\{1,\ldots,n+r\}.
$$

Consider three disjoint sets
\begin{equation*}
\mathcal{A}=\{1,\ldots ,r\},\quad \mathcal{B}=\{r+1,r+2,\ldots ,r+s\},\quad
\mathcal{C}=\{r+s+1,r+s+2,\ldots ,r+s+t\}.
\end{equation*}
Denote $r+s+t=n$ and consider the set $\mathcal{S}_{n}$ of permutations of $
\{1,\ldots ,n\}$. We will be interested in the subset $D$ of permutations from  $\mathcal{S}_n$ for which there are exactly $i$ elements from the set $
\mathcal{C}$ at the first $k$ positions and there are exactly $j$
elements from the set $\mathcal{A}$ at the first $\ell+k$ positions. That is,
\begin{equation*}
D=\{\sigma \in \mathcal{S}_{n}:\,\left\vert \sigma (\{1,\ldots ,k\})\cap
\mathcal{C}\right\vert =i\;\mbox{and}\;\left\vert \sigma (\{1,\ldots
,k+\ell\})\cap \mathcal{A}\right\vert =j\}.
\end{equation*}
We assume $i\leq \min \{t,k\}$ and $j\leq \min \{r,k+\ell\}$, since otherwise $
D=\emptyset $.

\begin{lemma}
	\label{DD} Let $\mathfrak{D}_{r,s,t,k,\ell,i,j}=|D|$, the number of elements in $D$. Then
	\begin{equation}
	\mathfrak{D}_{r,s,t,k,\ell,i,j}=\tfrac{n!}{\binom{n}{k,\ell}}\,\binom{t}{i}
	\binom{r}{j}\,\sum_{m=\max\{0, j-\ell\}}^{\min\{j, k-i\}}\,\binom{j}{m}\,\binom{s}{
		k-i-m}\,\binom{s+t+m-k}{\ell+m-j} . \label{numD}
	\end{equation}
\end{lemma}

\begin{proof}
	We denote $(a)_b=a(a-1)\ldots (a-b+1)$, where $b$ is positive integer, and $
	(a)_0=1$. Moreover, we follow the rule: $\binom{a}{b}=0$ if $b<0$ or $a<b$.
	
To obtain $\sigma\in D$ we perform the following four steps:
\begin{enumerate}
	\item Choose $i$ positions out of $\{1,\ldots ,k\}$ in $\binom{k}{i}$ ways
	and fill these positions with elements from $\mathcal{C}$ in $(t)_{i}$ ways.
	
	\item For any $m=0,\ldots ,j$ choose $m$ out of remaining $k-i$ positions
	in $\{1,\ldots ,k\}$ in $\binom{k-i}{m}$ ways and fill them with elements of
	$\mathcal{A}$ in $(r)_{m}$ ways. Remaining $k-i-m$ positions out of $\{1,\ldots
	,k\}$ fill with elements of $\mathcal{B}$ in $(s)_{k-i-m}$ ways.
	
	\item Choose $j-m$ positions for elements of $\mathcal{A}$ from $
	\{k+1,\ldots ,k+\ell\}$ in $\binom{\ell}{j-m}$ ways and fill them with elements of
	$\mathcal{A}$ in $(r-m)_{j-m}$ ways. Remaining $\ell-j+m$ positions out of $
	\{k+1,\ldots ,\ell\}$ fill with elements of $\mathcal{B}\cup \mathcal{C}$ in $
	(s-k+i+m+t-i)_{\ell-j+m}=(s+t+m-k)_{\ell-j+m}$ ways.
	
	\item The remaining $n-k-\ell$ positions fill with the rest of the elements of $
	\mathcal{A}\cup \mathcal{B}\cup \mathcal{C}$ in $(n-k-\ell)!$ ways.
	\end{enumerate}

	Combining these four steps we get
	\begin{equation*}
	|D|=\binom{k}{i}(t)_{i}\left( \sum_{m=0}^{j}\,\binom{k-i}{m}
	\,(r)_{m}\,(s)_{k-i+m}\,\binom{\ell}{j-m}\,(r-m)_{j-m}\,(s+t+m-k)_{\ell-j+m}
	\right) (n-k-\ell)!.
	\end{equation*}
	The formula \eqref{numD} follows by simple transformations involving e.g. $(r)_{m}\,(r-m)_{j-m}=(r)_{j}$.
\end{proof}

\begin{remark}
	\label{Alter} Since the subset of permutations $D$ as defined above can
	alternatively be written as
	\begin{equation*}
	D=\{\sigma \in \mathcal{S}_{n}:\,|\sigma (\{k+\ell+1,\ldots ,n\})\cap \mathcal{A
	}|=r-j\;\mbox{and}\;|\sigma (\{k+1,\ldots ,n\})\cap \mathcal{C}|=t-i\},
	\end{equation*}
	we have an equivalent formula for the number of elements in $D:$
	\begin{equation}
	|D|=\mathfrak{D}_{t,s,r,n-k-\ell,\ell,r-j,t-i}.  \label{alter}
	\end{equation}
\end{remark}

In the next result we give explicit forms for $\mathbb{P} (X_{i:m}=X_{k:n+r},\,X_{j:n}^{(r)}=X_{\ell:n+r})$ for all possible configurations of parameters $i,m,k,n,r,j,\ell$.

\begin{proposition}
	\label{prawd} Let $A=\{1,\ldots ,r\}$, $B=\{r+1,\ldots ,m\}$ and $
	C=\{m+1,\ldots ,n+r\}$. Probabilities
$$
p_{r,(i,m,k),(j,n,\ell)}:=\mathbb{P} (X_{i:m}=X_{k:n+r},\,X_{j:n}^{(r)}=X_{\ell:n+r}).
$$
are non-zero only if $i\leq k\leq i+n+r-m$ and $j\leq \ell\leq j+r$. Then
	\begin{itemize}
		\item[(i)] for $k<\ell$
		\begin{equation*}
		 p_{r,(i,m,k),(j,n,\ell)}=\tfrac{(n-j+1)(|A|\mathfrak{D}_{|A|-1,|B|,|C|,k-1,\ell-k-1,k-i,\ell-j-1}+|B|\mathfrak{D}_{|A|,|B|-1,|C|,k-1,\ell-k-1,k-i,\ell-j})}{(r+n-\ell+1)(r+n)!};
		\end{equation*}
		
		\item[(ii)] for $k=\ell$
		\begin{equation*}
		p_{r,(i,m,k),(j,n,k)}=\tfrac{|B|\mathfrak{D}_{|A|,|B|-1,|C|,k-1,0,k-i,k-j}}{
			(r+n)!};
		\end{equation*}
		
		\item[(iii)] for $k>\ell$
		\begin{equation*}
		 p_{r,(i,m,k),(j,n,\ell)}=\tfrac{(m-i+1)(|C|\mathfrak{D}_{|C|-1,|B|,|A|,\ell-1,k-\ell-1,\ell-j,k-i-1}+|B|\mathfrak{D}_{|C|,|B|-1,|A|,\ell-1,k-\ell-1,\ell-j,k-i})}{(r+n-k+1)(r+n)!}.
		\end{equation*}
	\end{itemize}
\end{proposition}
Proof of Prop. \ref{prawd}, due to its computational complexity, is given in  Appendix.

Now we are ready to derive the formula for the density of $f_{X_{i:m},X_{j:n}^{(r)}}$.

The independence property given in Lem. \ref{140407-1920} allows to write the density of $(X_{i:m},X_{j:n}^{(r)})$ as a linear combination of densities of bivariate os's from the sample $(X_1,\ldots,X_{n+r})$.
\begin{theorem}\label{mmain}
\begin{equation}
f_{X_{i:m},X_{j:n}^{(r)}}(x,y)=\,\sum_{k,\ell =1}^{n+r}\,p_{r,(i,m,k),(j,n,\ell)}\,f_{k,\ell :n+r}(x,y)
\label{eq08}
\end{equation}
with coefficients $p_{r,(i,m,k),(j,n,\ell)}$ as given in Prop. \ref{prawd}.
\end{theorem}
\begin{proof}Note that
	$$
	\P(X_{i:m}\le x,\,,X_{j:n}^{(r)}\le y)=\sum_{k,\ell =1}^{n+r}\,\mathbb{P} (X_{i:m}=X_{k:n+r},\,X_{j:n}^{(r)}=X_{\ell:n+r},\,X_{k:n+r}\le x,\,X_{\ell:n+r}\le y).
	$$
	From the proof of Prop. \ref{prawd} (see Appendix) it follows that the event $\{X_{i:m}=X_{k:n+r},\,X_{j:n}^{(r)}=X_{\ell:n+r}\}$ is a union of events of the form $\{X_{\sigma(1)}\le \ldots\le X_{\sigma(n+r)}\}$, where the union is with respect to permutations from special subsets of $\mathcal{S}_{n+r}$ (these subsets are different in each of three cases: $k<\ell$, $k=\ell$ and $k>\ell$). By Lem. \ref{140407-1920} it follows that
	$$
	\mathbb{P} (X_{i:m}=X_{k:n+r},\,X_{j:n}^{(r)}=X_{\ell:n+r},\,X_{k:n+r}\le x,\,X_{\ell:n+r}\le y)$$$$=	\mathbb{P} (X_{i:m}=X_{k:n+r},\,X_{j:n}^{(r)}=X_{\ell:n+r})\,\P(X_{k:n+r}\le x,\,X_{\ell:n+r}\le y).
	$$
	Therefore the density $f_{X_{i:m},X_{j:n}^{(r)}}$ of $(X_{i:m},X_{j:n}^{(r)})$ with respect to the measure $\nu$ (introduced in Section 1) assumes the form
	$$
	f_{X_{i:m},\,X_{j:n}^{(r)}}=\sum_{k,\ell =1}^{n+r}\,\mathbb{P} (X_{i:m}=X_{k:n+r},\,X_{j:n}^{(r)}=X_{\ell:n+r})\,f_{k,\ell:n+r}.
	$$	

Now the result follows by inserting in the above expression correct forms of probabilities $\mathbb{P} (X_{i:m}=X_{k:n+r},\,X_{j:n}^{(r)}=X_{\ell:n+r})$ which are given in Prop. \ref{prawd}
\end{proof}

Below we derive joint densities (with respect to $\nu$) of $(X_{i:m},X_{j:n}^{(r)})$ in several cases of special interest.

(i) \textbf{Order statistics from the original and extended samples.}
Without any loss of generality we can assume that $m\leq n$. Since
\begin{equation}
\mathbb{P}(X_{i:m}=X_{k:n})=\mathbb{P}(X_{i:m}=X_{k:n},\,X_{j:n}^{(0)}=X_{j:n}),
\label{eq07}
\end{equation}
Prop. \ref{prawd} with $r=0$, $j=\ell$, $|A|=0$, $|B|=m$, $|C|=n-m$ applies and since the left hand side of \eqref{eq07} does not depend on $\ell $ we can choose  the case $k=\ell$.
Therefore,
\begin{eqnarray*}
\mathbb{P}(X_{i:m}=X_{k:n})=\tfrac{m\mathfrak{D}_{0,m-1,n-m,k-1,0,k-i,0}}{
n!}=\tfrac{m(n-k)!(k-1)!}{n!}\,\binom{n-m}{k-i}\binom{m-1}{i-1}=\tfrac{\binom{k-1}{i-1}\binom{n-k}{m-i}}{\binom{n}{m}},
\end{eqnarray*}
and thus the formula for the density of $(X_{i:m},\,X_{j:n})$ agrees with \eqref{140410-1856}.

(ii) \textbf{Moving maxima. }We consider $(X_{n:n},X_{n:n}^{(r)})$. From
Prop. \ref{prawd} we get
\begin{eqnarray*}
&&\mathbb{P}(X_{n:n} =X_{k:n+r},\,X_{n:n}^{(r)}=X_{n+r:n+r})=\tfrac{\binom{k-1}{
n-1}}{\binom{n+r}{r}},\quad k=n,n+1,\ldots ,n+r-1, \\
&&\mathbb{P}(X_{n:n} =X_{n+r:n+r},\,X_{n:n}^{(r)}=X_{n+r:n+r})=\tfrac{n-r}{n+r},
\\
&&\mathbb{P}(X_{n:n} =X_{n+r:n+r},\,X_{n:n}^{(r)}=X_{\ell:n+r})=\tfrac{\binom{\ell-1}{n-1}}{\binom{n+r}{r}},\quad \ell=n,n+1,\ldots ,n+r-1.
\end{eqnarray*}
Consequently, Th. \ref{mmain} gives
$$
f_{X_{n:n},X_{n:n}^{(r)}}(x,y)=\left\{\begin{array}{ll} \sum_{k=n}^{n+r-1}\,\tfrac{\binom{k-1}{n-1}}{\binom{n+r}{r}}\,f_{k,n+r:n+r}(x,y), & x<y, \\
\tfrac{n-r}{n+r}\,f_{n+r:n+r}(x), & x=y, \\
\sum_{k=n}^{n+r-1}\,\tfrac{\binom{k-1}{n-1}}{\binom{n+r}{r}}\,f_{k,n+r:n+r}(y,x), & x>y.\end{array}\right.
$$
(iii) \textbf{Moving minima. }\label{mmi} We consider $(X_{1:n},X_{1:n}^{(r)})$. Then from Prop. \ref{prawd} we get
\begin{eqnarray*}
&&\mathbb{P}(X_{1:n} =X_{1:n+r},\,X_{1:n}^{(r)}=X_{\ell:n+r})=\tfrac{\binom{n+r-\ell}{n-1}}{\binom{n+r}{r}},\quad \ell=2,3,\ldots ,r+1, \\
&&\mathbb{P}(X_{1:n} =X_{1:n+r},\,X_{1:n}^{(r)}=X_{1:n+r})=\tfrac{n-r}{n+r}, \\
&&\mathbb{P}(X_{1:n} =X_{k:n+r},\,X_{1:n}^{(r)}=X_{1:n+r})=\tfrac{\binom{n+r-k}{n-1}}{\binom{n+r}{r}},\quad k=2,3,\ldots ,r+1.
\end{eqnarray*}
Consequently, Th. \ref{mmain} gives
$$
f_{X_{1:n},X_{1:n}^{(r)}}(x,y)=\left\{\begin{array}{ll} \sum_{k=2}^{r+1}\,\tfrac{\binom{n+r-k}{n-1}}{\binom{n+r}{r}}\,f_{1,k:n+r}(x,y), & x<y, \\
\tfrac{n-r}{n+r}\,f_{1:n+r}(x), & x=y, \\
\sum_{k=2}^{r+1}\,\tfrac{\binom{n+r-k}{n-1}}{\binom{n+r}{r}}\,f_{1,k:n+r}(y,x), & x>y.\end{array}\right.
$$
(iv) \textbf{Moving $i$th os's.} We consider $(X_{i:m},\,X_{i:m}^{(1)})$.
Then from Prop. \ref{prawd} we get
\begin{eqnarray*}
&&\mathbb{P}(X_{i:m}=X_{i:m+1},\,X_{i:m}^{(1)}=X_{i:m+1})=
\tfrac{(m-i+1)(m-i)}{(m+1)m}, \\
&&\mathbb{P}(X_{i:m}=X_{i:m+1},\,X_{i:m}^{(1)}=X_{i+1:m+1})=
\tfrac{i(m-i+1)}{(m+1)m}, \\
&&\mathbb{P}(X_{i:m}=X_{i+1:m+1},\,X_{i:m}^{(1)}=X_{i:m+1})=
\tfrac{i(m-i+1)}{(m+1)m}, \\
&&\mathbb{P}(X_{i:m}=X_{i+1:m+1},\,X_{i:m}^{(1)}=X_{i+1:m+1})=
\tfrac{i(i-1)}{(m+1)m}.
\end{eqnarray*}
Consequently, Th. \ref{mmain} gives
$$
f_{X_{i:m},X_{i:m}^{(1)}}(x,y)=\left\{\begin{array}{ll} \tfrac{i(m-i+1)}{(m+1)m}\,f_{i,i+1:m+1}(x,y), & x<y,\\
\tfrac{(m-i+1)(m-i)}{(m+1)m}\,f_{i:m+1}(x)+\tfrac{i(i-1)}{(m+1)m}\,f_{i+1:m+1}(x), & x=y,\\
\tfrac{i(m-i+1)}{(m+1)m}\,f_{i,i+1:m+1}(y,x), & x>y.\end{array}\right.$$

\section{Regression of overlapping os's}\label{Sec3}
From Prop. \ref{prawd} we know that $p_{r,(i,m,k),(j,n,l)}$ are non-zero
only if $i\leq k\leq i+n+r-m$ and $j\leq l\leq j+r$. This together with (\ref{eq08}) implies
\begin{equation*}
f_{X_{i:m},X_{j:n}^{(r)}}(x,y)=\sum_{k=i}^{n+r-m+i}\,\sum_{\ell
=j}^{j+r}\,p_{r,(i,m,k),(j,n,l)}\,f_{k,\ell :n+r}(x,y).
\end{equation*}

\noindent Consequently, the conditional distribution $\P_{X_{i:m}|X_{j:n}^{(r)}=y}$ has a density with respect to $\nu_y(dx)=\mu_1(dx)+\delta_y(dx)$ of the form
$$
f_{X_{i:m}|X_{j:n}^{(r)}=y}(x)=\sum_{k=i}^{n+r-m+i}\,\sum_{\ell
=j}^{j+r}\,p_{r,(i,m,k),(j,n,l)}\,f_{X_{k:n+r}|X_{\ell :n+r}=y}(x)\,\tfrac{f_{\ell:n+r}(y)}{f_{j:n}(y)}$$$$=\sum_{k=i}^{n+r-m+i}\,\sum_{\ell
=j}^{j+r}\,p_{r,(i,m,k),(j,n,l)}\,\tfrac{\ell\binom{n+r}{\ell}}{j\binom{n}{j}}\,F^{\ell -j}(y)\bar{F}
^{j+r-\ell }(y)\,f_{X_{k:n+r}|X_{\ell :n+r}=y}(x).
$$
and the conditional distribution $\P_{X_{j:n}^{(r)}|X_{i:m}=x}$  has a density with respect to $\nu_x(dy)=\mu_1(dy)+\delta_x(dy)$ of the form
$$
f_{X_{j:n}^{(r)}|X_{i:m}=x}(y)=\sum_{k=i}^{n+r-m+i}\,\sum_{\ell
=j}^{j+r}\,p_{r,(i,m,k),(j,n,l)}\,f_{X_{\ell :n+r}|X_{k:n+r}=x}(y)\,\tfrac{f_{k:n+r}(x)}{f_{i:m}(x)}$$$$=\sum_{k=i}^{n+r-m+i}\,\sum_{\ell
=j}^{j+r}\,p_{r,(i,m,k),(j,n,l)}\,\tfrac{k\binom{n+r}{k}}{i\binom{m}{i}}\,F^{k -i}(x)\bar{F}
^{n+r-m-k+i }(x)\,f_{X_{\ell :n+r}|X_{k:n+r}=x}(y).
$$

Therefore,
\begin{equation}\label{ce} \E(X_{i:m}|X_{j:n}^{(r)}=y)=\sum_{k=i}^{n+r-m+i}\,\sum_{\ell =j}^{j+r}\,p_{r,(i,m,k),(j,n,l)}\,
\tfrac{\ell\binom{n+r}{\ell}}{j\binom{n}{j}}\,F^{\ell -j}(y)\bar{F}
^{j+r-\ell }(y)\,\mathbb{E}(X_{k:n+r}|X_{\ell :n+r}=y)
\end{equation}
and
\begin{equation}\label{ec}
 \E(X_{j:n}^{(r)}|X_{i:m}=x)=\sum_{k=i}^{n+r-m+i}\,\sum_{\ell =j}^{j+r}\,p_{r,(i,m,k),(j,n,l)}\,
\,\tfrac{k\binom{n+r}{k}}{i\binom{m}{i}}\,F^{k -i}(x)\bar{F}
^{n+r-m-k+i }(x)\,\mathbb{E}(X_{\ell :n+r}|X_{k:n+r}=x).
\end{equation}

 That is, both regressions we are interested in are represented through rather complicated expressions \eqref{ce} and \eqref{ec}. Thus  characterizations or identifiability questions for  parent distributions through the form of $\mathbb{E}(X_{i:m}|X_{j:n}^{(r)})$  or $\mathbb{E}(X_{j:n}^{(r)}|X_{i:m})$ seems to be a difficult task in such a general framework. Therefore we will concentrate rather on the special cases of $r=0$ (when $m<n$) distinguishing two quite different subcases: in Subsection \ref{mn} we will consider characterizations by linearity of $\mathbb{E}(X_{i:m}|X_{j:n})$ while in Subsection \ref{nm} we will study identification through $\mathbb{E}(X_{j:n}|X_{i:m})$. For $r>0$ we will consider only the simplest case of $r=1$ and $m=n=2$ in Subsection \ref{mnr} below.

 \subsection{Identifiability through regression functions when $r=1$ and $m=n=2$}\label{mnr}
 Here we only consider the simplest case of os's from overlapping samples $(X_{1},X_{2})$ and $(X_{2},X_{3})$, that is the case of
$r=1$, $m=n=2$. Then \eqref{ce} gives
\begin{enumerate}
\item[(i)] $\mathbb{E}(X_{2:2}|X_{2:2}^{(1)}=y)=\tfrac{y}{2}F(y)+\int_{y}^{
\infty }\,xf(x)\,dx+\tfrac{1}{F(y)}\,\int_{-\infty }^{y}\,xF(x)f(x)\,dx,$

\item[(ii)] $\mathbb{E}(X_{1:2}|X_{1:2}^{(1)}=y)=\tfrac{y}{2}\bar{F}
(y)+\int_{-\infty }^{y}\,xf(x)\,dx+\tfrac{1}{\bar{F}(y)}\,\int_{y}^{\infty
}\,x\bar{F}(x)f(x)\,dx,$

\item[(iii)] $\mathbb{E}(X_{1:2}|X_{2:2}^{(1)}=y)=\tfrac{y}{2}\bar{F}(y)+\tfrac{
1}{2F(y)}\,\int_{-\infty }^{y}\,xf(x)\,dx+\int_{-\infty }^{y}\,xf(x)\,dx-
\tfrac{1}{F(y)}\,\int_{-\infty }^{y}\,xF(x)f(x)\,dx,$

\item[(iv)] $\mathbb{E}(X_{2:2}|X_{1:2}^{(1)}=y)=\tfrac{y}{2}F(y)+\tfrac{1}{2
\bar{F}(y)}\,\int_{y}^{\infty }\,xf(x)\,dx+\int_{y}^{\infty }\,xf(x)\,dx-
\tfrac{1}{\bar{F}(y)}\,\int_{y}^{\infty }\,x\bar{F}(x)f(x)\,dx.$
\end{enumerate}

We will show that each of these four regressions
determines uniquely the parent distribution. Note  that for $Y_{i}=-X_{i}$ and $u=-y$ we have $\mathbb{E}
(X_{1:2}|X_{1:2}^{(1)}=y)=-\mathbb{E}(Y_{2:2}|Y_{2:2}^{(1)}=u)$ and $\mathbb{E}
(X_{1:2}|X_{2:2}^{(1)}=y)=-\mathbb{E}(Y_{2:2}|Y_{1:2}^{(1)}=u)$. Consequently, (i) and (ii) as well as (iii) and (iv) above are
equivalent.

\begin{theorem}
\label{2222} Let the parent distribution be absolutely continuous distribution with the
interval support $(a,b)$. Then regression function $\mathbb{E}
(X_{2:2}|X_{2:2}^{(1)}=y)$ (alternatively, $\mathbb{E}(X_{1:2}|X_{1:2}^{(1)}=y)$), $y\in(a,b)$,
determines uniquely the distribution of $X$ if $(a,b)\subsetneq {\mathbb{R}}$. If $(a,b)={\mathbb{R}}$ it determines the distribution up to a shift.
\end{theorem}

\begin{proof}
It suffices to consider only the case of $\mathbb{E}(X_{2:2}|X_{2:2}^{(1)})$.

Let us denote $K(y)=\mathbb{E}(X_{2:2}|X_{2:2}^{(1)}=y)$, $y\in (a,b)$.
Performing integration by parts in (i)  we get
\begin{equation*}
K(y)=y+\int_{y}^{b}\,\bar{F}(x)\,dx-\tfrac{1}{2F(y)}\int_{a}^{y}\,F^{2}(x)\,dx.
\end{equation*}
Consequently, if $K(y)$ is the same for two distribution functions $F$ and $G
$, which are strictly increasing on $(a,b)$ and thus have differentiable
quantile functions $Q_F$ and $Q_G$, then
\begin{equation*}
2t\int_{t}^{1}\,(1-w)\,dQ_F(w)-\int_{0}^{t}\,w^{2}\,dQ_F(w)=2t
\int_{t}^{1}\,(1-w)\,dQ_G(w)-\int_{0}^{t}\,w^{2}\,dQ_G(w).
\end{equation*}
For $H:=(Q_F-Q_G)^{\prime }$  we obtain
\begin{equation}
L:=2t\int_{t}^{1}\,(1-w)\,H(w)\,dw=\int_{0}^{t}\,w^{2}\,H(w)\,dw=:R,\qquad t\in
(0,1).  \label{2tH}
\end{equation}
Differentiating \eqref{2tH} with respect to $t$ twice we obtain
\begin{equation*}
4(1-t)H(t)+t(2-t)H^{\prime }(t)=0,\qquad t\in (0,1).
\end{equation*}
Consequently, there exists a constant $c$ such that
\begin{equation*}
H(t)=\tfrac{c}{(2-t)^{2}t^{2}},\qquad t\in (0,1).
\end{equation*}
Now, assuming that $c\neq 0$ we plug $H$ back into \eqref{2tH}. After
canceling $c$ at the left hand side we get
\begin{equation*}
L=2t\int_{t}^{1}\,\tfrac{1-w}{w^{2}(2-w)^{2}}\,dw<2t\int_{t}^{1}\,w^{-2}
\,dw=2(1-t),
\end{equation*}
while at the right hand side we have
\begin{equation*}
R=\int_{0}^{t}\,(2-w)^{-2}\,dw>\frac{t}{4}.
\end{equation*}
Therefore, for any $t\in (0,1)$ we have $t<8(1-t)$,
which is impossible for $t$ sufficiently close to 1. Consequently, $c=0$,
and thus $Q_F-Q_G=B$ for some constant $B$. It implies that either $B=0
$ or $a=-\infty $ and $b=\infty $. In the latter case for any $y\in {\mathbb{%
R}}$ there exists unique $x\in {\mathbb{R}}$ such that $y=Q_G(F(x))$.
Therefore
\begin{equation*}
G(y)=F(x)=F(Q_F(F(x)))=F(Q_G(F(x))+B)=F(y+B).
\end{equation*}
This completes the proof.
\end{proof}

\begin{theorem}
\label{2223} Let $X$ has absolutely continuous distribution with the
interval support $(a,b)$. Then the regression function $\mathbb{E}
(X_{1:2}|X_{2:2}^{(1)}=y)$ (alternatively, $\mathbb{E}(X_{2:2}|X_{1:2}^{(1)}=y)$, $y\in(a,b)$,
determines uniquely the distribution of $X$ if $(a,b)\subsetneq {\mathbb{R}}$. If $(a,b)={\mathbb{R}}$ it determines the distribution up to a shift.
\end{theorem}

\begin{proof}
Again we consider only the case of $\mathbb{E}(X_{1:2}|X_{2:2}^{(1)})$.

Let us denote $K(y)=\mathbb{E}(X_{1:2}|X_{2:2}^{(1)}=y)$, $y\in (a,b)$.
Performing integration by parts in (iii) we get
\begin{equation*}
K(y)=y-\left( \tfrac{1}{2F(y)}+1\right) \,\int_{a}^{y}\,F(x)\,dx+\tfrac{1}{
2F(y)}\int_{a}^{y}\,F^{2}(x)\,dx.
\end{equation*}
Consequently, if $K(y)$ is the same for two distribution functions $F$ and $G
$, which are strictly increasing on $(a,b)$ and thus have differentiable
inverses $Q_F$ and $Q_G$, then from the above formula we get
\begin{equation*}
(1+2t)\int_{0}^{t}\,w\,dQ_F(w)-\int_{0}^{t}\,w^{2}\,dQ_F(w)=(1+2t)
\int_{0}^{t}\,w\,dQ_G(w)-\int_{0}^{t}\,w^{2}\,dQ_G(w).
\end{equation*}
As in the proof above we denote $H=(Q_F-Q_G)^{\prime }$. Then we have
\begin{equation}
L:=(1+2t)\int_{0}^{t}\,w\,H(w)\,dw=\int_{0}^{t}\,w^{2}\,H(w)\,dw=:R,\qquad t\in
(0,1).  \label{(1+2t)H}
\end{equation}
Differentiating \eqref{(1+2t)H} twice with respect to $t$ we obtain
\begin{equation*}
(1+4t)H(t)+t(1+t)H^{\prime }(t)=0,\qquad t\in (0,1).
\end{equation*}
Consequently, there exists a constant $c$ such that
\begin{equation*}
H(t)=\tfrac{c}{t(1+t)^{3}},\qquad t\in (0,1).
\end{equation*}
Now, assuming that $c\neq 0$ we plug $H$ back into \eqref{(1+2t)H}. After
canceling $c$ at the left hand side we have
\begin{equation*}
L=(1+2t)\int_{0}^{t}\,(1+w)^{-3}\,dw=\tfrac{1+2t}{2}\left( 1-\tfrac{1}{
(1+t)^{2}}\right) =\tfrac{(1+2t)((1+t)^{2}-1)}{2(1+t)^{2}}
\end{equation*}
while at the right hand side we have
\begin{equation*}
R=\int_{0}^{t}\,\tfrac{w}{(1+w)^{3}}\,dw=\tfrac{t^{2}}{2(1+t)^{2}}.
\end{equation*}
Obviously, $L\neq R$. Consequently, $c=0$, and thus $Q_F-Q_G=B$ for
some constant $B$.
To complete the proof we proceed as in the end of the proof of the previous theorem.
\end{proof}

\subsection{Linearity of regression of $X_{i:m}$ given $X_{j:n}$, $m<n$}\label{mn}

In this subsection we consider linearity of regression as given in \eqref{eq03} when $i<j<n-m+i$ since, as mentioned before, the cases $j\le i$ and $j\ge n-m+i$ have already been discussed in DW.

It seems that the only results available  in the literature are for $m=i=1$, i.e. with $X_{i:m}=X_1$. In particular, the case $n=2k+1$, $j=k+1$ was considered in Weso\l owski and Gupta (2001) (referred to by WG in the sequel), see also Nagaraja and Nevzorov (1997). In WG it was shown  that if $\E\,X_{1}=0$ the relation
\begin{equation}
\mathbb{E}(X_{1:1}|X_{k+1:2k+1})=aX_{k+1:2k+1},
\label{eq04}
\end{equation}
implies $a\ge \tfrac{k+1}{2k+1}$ and up to a scaling factor uniquely determines the parent distribution. Examples of such distributions (up to scale factors) are:
\begin{itemize}
\item  The uniform distribution on $[-1,1]$ when $a=\tfrac{k+1}{2k+1}$.
\item The Student $t_2$ distribution when $a=1$. In fact, the same characterization of the Student $t_2$ distribution, sometimes with a different phrasing, e.g. writing $\tfrac{1}{n}\sum_{j=1}^{n}\,X_{j}$ instead of $X_{1:1}=X_{1}$, under the conditional
expectation in \eqref{eq04} with $a=1$ is given in Nevzorov (2002) and Nevzorov, Balakrishnan and Ahsanullah (2003). (These references apparently missed the result of WG.)

\item The distribution with the cdf $F(x)=\tfrac{1}{2}\left(1+\tfrac{\sqrt{2}x}{\sqrt{\sqrt{4+x^{4}}+x^{2}}}\right)
$, $x\in {\mathbb{R}}$, when $a=\tfrac{4k+3}{3(2k+1)}.$
\end{itemize}

Related regression characterizations can be found also e.g. in  Balakrishnan and Akhundov (2003), Akhundov, Balakrishnan and Nevzorov (2004) (referred to by ABN in the sequel), Nevzorova, Nevzorov and Akhundov (2007), Marudova and Nevzorov (2009), Akhundov and Nevzorov (2012), Yanev and Ahsanullah (2012).

For describing our results in the sequel (as well as results from the literature) it is convenient to introduce the family of complementary beta distributions  defined in Jones (2002) for a restricted range of parameters $\alpha, \beta$ and then extended to any $\alpha,\beta\in\R$ in Jones (2007). Slightly changing the original Jones' formulation we say that an absolutely continuous distribution with distribution function $F$ and density $f$ belongs to the family of complementary beta distributions  $CB(\alpha,\beta)$, $\alpha,\beta\in\R$, if
$$
F^{\alpha}(x)\,\bar{F}^{\beta}(x)\propto f(x),\quad x\in \mathbb{R}.
$$
Kamps (1991) introduced $CB(-p,1+p-q)$ for integer $p$ and $q$ in the context of characterization of distributions by
recurrence relations between moments of os's from the original and
extended samples. Nevzorov, Balakrishnan and Ahsanullah (2003) observed that $CB(\alpha ,\alpha )$ includes several interesting
special cases: Student $t_{2}$ distribution for $\alpha =3/2$, logistic
distribution for $\alpha =1$, squared sine distribution with $F(x)=\sin
^{2}(x)I_{[0,\pi /2]}(x)$ for $\alpha=1/2$. Extensive discussion of properties of $CB(\alpha,\beta)$ family is given in Jones (2007). In particular, it is observed there that if $\alpha=0$ then $\beta=1$ gives the exponential distribution, $\beta>1$ gives the Pareto laws and $\beta<1$ are power distributions on a bounded interval.

The family $CB(\alpha,\beta)$ is a subclass of the family  $GS(\alpha,\beta,\gamma)$, the latter being defined through the equation $$F^{\alpha}(x)(1-F^{\gamma}(x))^{\beta}\propto f(x),$$ was introduced in Mui\~no, Voit and Sorribas (2006), i.e. $CB(\alpha,\beta)=GS(\alpha,\beta,1)$. Another subclass of $GS$ distributions  was characterized by linearity of regression of sum of $X_{k-j:k-j}$ and $X_{k+r:k+r}$ given $X_{k:k}$ in Marudova and Nevzorov (2009). Basic properties of os's from a sample with the parent distribution belonging to  the GS family were studied in Mohie El-Din, Amein and Hamedani (2012). In a recent paper Hu and Lin (2018) considered an extension of $GS(1,2,\gamma)$ class defined by the equation $$F(x)(1-F^{\gamma}(x))\propto x^af(x),$$ for $\gamma>0$ and $a\in[0,1]$ in the context of characterization by random exponential shifts of os's. All these equations can be treated as generalized versions of the logistic growth model or even as a  more flexible growth model introduced in Richards (1959).

To describe results in this section it will be also convenient to use the quantile function $Q$ and quantile density function $q$. If $f$ is a  strictly positive density on some (possibly unbounded) interval $(a,b)$ then  the  respective  distribution function $F$ is invertible on $(a,b)$ and thus its inverse, quantile function $Q$, is well defined on $(0,1)$. Moreover it is absolutely continuous with respect to the Lebesgue measure on $(0,1)$, that is $Q(y)=\int_{y_0}^y\,q(u)\,du$ for some $y_0\in[0,1]$. The function $q$ is called the quantile density function. Note that
\begin{equation}\label{fq}
f=T(F)\qquad \Leftrightarrow \qquad q=1/T,
\end{equation}
and thus $q$ together with $y_0$ uniquely determine $F$. In particular, the quantile density $q$ for a distribution in $CB(\alpha,\beta)$ has the form
$$
q(u)\propto u^{-\alpha}(1-u)^{-\beta},\quad u\in(0,1).
$$


\begin{remark}
Note that the regression condition
\begin{equation}
\label{11jn}
\E(X_{1:1}|X_{j:n})=aX_{j:n},
\end{equation}
has been reduced in WG to the condition $M_{\lambda}(x)=Ax$ where $A=\tfrac{na-1}{n-1}$
and $$M_{\lambda}(x):=\lambda\mathbb{E}(X|X<x)+(1-\lambda)\mathbb{E}(X|X>x)$$ with $\lambda=\tfrac{j-1}{n-1}$. In particular, formula (4) in WG says that  for a positive $A$ (necessarily, $A\ge 1/2$) condition \eqref{11jn} holds if and only if the quantile function $Q$ satisfies
\begin{equation} \label{gwq}
Q(y)=cy^{\tfrac{\lambda}{A}-1}\,
(1-y)^{\tfrac{1-\lambda}{A}
	-1}\,(\lambda-y),\quad y\in(0,1).
\end{equation}
Differentiating \eqref{gwq} we get the following expression for the quantile density
\begin{equation*}
q(y)\propto y^{\tfrac{\lambda}{A}-2}(x)(1-y)^{\tfrac{1-\lambda}{A}-2}((\lambda-A)\lambda-2\lambda(1-A)y+(1-A)y^2).
\end{equation*}
Therefore, for $A=1$
$$
q(y)\propto y^{-1-\tfrac{n-j}{n-1}}(1-y)^{-1-\tfrac{j-1}{n-1}},\qquad y\in(0,1).
$$
Since $A=1$ implies $a=1$ we conclude that $\E(X_{1:1}|X_{j:n})=X_{j:n}$ characterizes the $CB\left(1+\tfrac{n-j}{n-1},1+\tfrac{j-1}{n-1}\right)$ family. This result was independently proved in Balakrishnan and Akhundov (2003), see also Cor. 2.1 in ABN and Nevzorov (2015).
\end{remark}


ABN characterized also the family $CB(1+(1-\lambda )i,\,1+\lambda i)$ for
positive integer $i$ and $\lambda \in (0,1)$ through the condition
\begin{equation*}
\mathbb{E}(\lambda X_{i:2i+1}+(1-\lambda )X_{i+2:2i+1}|X_{i+1:2i+1})=X_{i+1:2i+1}.
\end{equation*}
This result, stated as Th. 3.1 in ABN, includes the result of Nevzorov (2002) who characterized the family $CB\left(1+\tfrac{i}{2},\,1+\tfrac{i}{2}\right)$ by the above condition with $\lambda =1/2$.

\begin{remark}
	\label{abn}
Before we proceed further with a new related characterization let us add here a small complement to Th. 3.1 of ABN (also valid for Th. 2 of Nevzorov (2002)).
The proof as given in ABN (similarly as that of the main result from Balakrishnan and Akhundov (2003)),  exploits the ideas originating from the proof of Th. 1 in Nevzorov (2002). In particular, an important part of the proof of Th. 3.1 in ABN is integration by parts in which the following two identities
\begin{equation}
\label{2lim}
\lim_{u\to-\infty}uF^i(u)=0\quad\mbox{and}\quad \lim_{u\to\infty}u\bar{F}^i(u)=0
\end{equation}
are necessary. Therefore either these conditions or suitable moments conditions, see Lem. \ref{iint} below, have to strengthen  the assumptions of Th. 3.1 of ABN. It is of some interest to note that  distributions from $CB(1+(1-\lambda)i,\,1+\lambda i)$ appearing in the statement of Theorem 3.4 given below do not possess finite expectations.
\end{remark}

The result below shows that  natural integrability assumptions are responsible for speed of convergence to zero of suitable powers of the left and right  tails of the cdf $F$. In particular, it follows that integrability of both $X_{i:2i+1}$ and $X_{i+2:2i+1}$ imply \eqref{2lim}.

\begin{lemma}\label{iint}
	If $\E|X_{k:n}|<\infty$ then
	\begin{equation}\label{limi}
	\lim_{x\to-\infty}\,xF^k(x)=0\quad and \quad\lim_{x\to\infty}x\bar{F}^{n-k+1}(x)=0.
	\end{equation}
\end{lemma}
\begin{proof}
	Note that, for $x<0$,
	$$\int_{-\infty}^{x}\,|t|\,F^{k-1}(t)\bar{F}^{n-k}(t)\,f(t) dt\ge |x|\bar{F}^{n-k}(x)\int_{-\infty}^x\,F^{k-1}(t)f(t)\,dt=\frac{1}{k}|z|\bar{F}^{n-k}(x)F^k(x).$$
	Integrability of $X_{k:n}$ implies that the left hand side above converges to 0 as $x\to-\infty$. The first limit in \eqref{limi} follows since $\lim_{x\to-\infty}\,\bar{F}^{n-k}(x)=1$.
	
	Similarly, for $x>0$,
	$$\int_x^{\infty}\,|t|\,F^{k-1}(t)\bar{F}^{n-k}(t)\,f(t) dt\ge |x|F(x)^{k-1}\int_x^{\infty}\,\bar{F}^{n-k}(t)f(t)\,dt=\frac{1}{n-k+1}|x|F^{k-1}(x)\bar{F}^{n-k+1}(x).$$
	Since $\E|X_{k:n}|<\infty$ the left hand side above converges to 0 as $x\to\infty$. Since $\lim_{x\to\infty}\,F^{k-1}(x)=1$ the second limit in \eqref{limi} follows.
	\end{proof}
Basically, we have described the state of art of the characterizations by linearity of regression  of an os from the original sample given an os from the extended sample.  In the next result we present a new contribution whose proof borrows some ideas from Nevzorov (2002).

\begin{theorem}\label{2.2}
	Let $2\le j\le n-1$. Assume that $\E|X_{j-1:n-2}|<\infty$ and that $f>0$ on some interval $(a,b)$ (possibly unbounded). If
	\begin{equation}
	\label{akuku}
	\E(X_{j-1:n-2}|X_{j:n})=X_{j:n},
	\end{equation}
	then the parent distribution has the quantile density $q$ of the form
	\begin{equation}
	\label{qdf}
	q(u)\propto \tfrac{j-1+(n-2j+1)u}{u^{1+(j-1)\lambda}\,(1-u)^{1+(n-j)(1-\lambda)}},\quad u\in(0,1),
	\end{equation}
	where $\lambda=\tfrac{j(j-1)}{(n-j+1)(n-j)+j(j-1)}$.
\end{theorem}
\begin{proof}
	Due to \eqref{expe1} we can write
	$$
	 \E(X_{j-1:n-2}|X_{j:n})=\tfrac{(n-j+1)(n-j)}{n(n-1)}\E(X_{j-1:n}|X_{j:n})+\tfrac{2(n-j)(j-1)}{n(n-1)}X_{j:n}+\tfrac{j(j-1)}{n(n-1)}\E(X_{j+1:n}|X_{j:n}).
	$$
	Combining the above equation with \eqref{akuku} we obtain
	\begin{equation}
	\label{akuk}
	(1-\lambda)\E(X_{j-1:n}|X_{j:n})+\lambda\E(X_{j+1:n}|X_{j:n})=X_{j:n}.
	\end{equation}
	
	Note that
	$$
	\E(X_{j-1:n}|X_{j:n}=x)=\int_{-\infty}^x\,t\tfrac{dF^{j-1}(t)}{F^{j-1}(x)}.
	$$
	Now, Lem. \ref{iint} implies $\lim_{x\to-\infty}\,xF^{j-1}(x)=0$ and thus integration by parts gives
	$$
	\E(X_{j-1:n}|X_{j:n}=x)=x-\tfrac{\int_{-\infty}^{x}\,F^{j-1}(t)\,dt}{F^{j-1}(x)}.
	$$
	On the other hand,
	$$
    \E(X_{j+1:n}|X_{j:n}=x)=-\int_x^{\infty}\,t\,\tfrac{d\bar{F}^{n-j}(t)}{\bar{F}^{n-j}(x)}.
	$$
	Lem. \ref{iint} implies $\lim_{x\to\infty}\,x\bar{F}^{n-j}(x)=0$ and thus integration by parts gives
	$$
	 \E(X_{j+1:n}|X_{j:n}=x)=x+\tfrac{\int_x^{\infty}\,\bar{F}^{n-j}(t)\,dt}{\bar{F}^{n-j}(x)}.
	$$
	
	Consequently, \eqref{akuk} assumes the form
	$$
	 \tfrac{\int_{-\infty}^{x}\,F^{j-1}(t)\,dt}{F^{j-1}(x)}=\tfrac{\lambda}{1-\lambda}\tfrac{\int_x^{\infty}\,\bar{F}^{n-j}(t)\,dt}{\bar{F}^{n-j}(x)}.
	$$
	
	Let us introduce two functions $G$ and $H$ defined as follows: $G(x)=\int_{-\infty}^{x}\,F^{j-1}(t)\,dt$ and $H(x)=\int_x^{\infty}\,\bar{F}^{n-j}(t)\,dt$. Consequently, the above equation can be written as
		\begin{equation}\label{de}
		\tfrac{G'}{G}=-\tfrac{1-\lambda}{\lambda}\tfrac{H'}{H}.
		\end{equation}
Upon integration we get $GH^{\tfrac{1-\lambda}{\lambda}}=K$ for some constant $K$. Then after multiplication of \eqref{de} by $GH^{1/\lambda}$ we get
$$
H^{1/\lambda}
=
\displaystyle
-\frac{1-\lambda}{\lambda}
\left( GH^{\frac{1-\lambda}{\lambda}} \right)\displaystyle\frac{H'}{G'}
=
K\frac{1-\lambda}{\lambda}
\frac{\overline{F}^{n-j}}{F^{j-1}}
$$
and thus
$$
H\propto F^{-(j-1)\lambda}\bar{F}^{(n-j)\lambda}.
$$
Now we differentiate the above equation and obtain
$$
f\propto \tfrac{F^{1+(j-1)\lambda}\,\bar{F}^{1+(n-j)(1-\lambda)}}{j-1+(n-2j+1)F}.
$$
Thus the final result follows from \eqref{fq}.
\end{proof}

Now we will present two corollaries of the above result which are closely connected to regression characterizations considered in literature.

\begin{corollary} Assume that $\E|X_{i:2i-1}|<\infty$. If
\begin{equation}
\mathbb{E}(X_{i:2i-1}|X_{i+1:2i+1})=X_{i+1:2i+1},
\label{exGW}
\end{equation}
then the parent distribution is $CB\left(1+\tfrac{i}{2},\,1+\tfrac{i}{2}\right)$.
\end{corollary}\begin{proof} It follows directly from Th. \ref{2.2} by taking there $j=i+1$ and $n=2i+1$. Note that then $\lambda=\tfrac{1}{2}$. \end{proof}

The above corollary is also a consequence of a characterization of  $CB(1+(1-\lambda )i,\,1+\lambda i)$ distribution for
positive integer $i$ and $\lambda \in (0,1)$ through the condition
\begin{equation}\label{abn2}
\mathbb{E}(\lambda X_{i:2i+1}+(1-\lambda )X_{i+2:2i+1}|X_{i+1:2i+1})=X_{i+1:2i+1},
\end{equation}
which is stated as Th. 3.1 in ABN (it also includes the result of Nevzorov (2002) who characterized the family $CB\left(1+\tfrac{i}{2},\,1+\tfrac{i}{2}\right)$ by the above condition with $\lambda =1/2$). Note that \eqref{exGW} combined with \eqref{expe1} gives \eqref{abn2} with $\lambda=1/2$.

The second corollary is related to an open problem stated in DW.
	 In the concluding remarks of that paper, the authors suggested that possibly the easiest open questions in characterizations by linearity of regression of an os from a restricted sample with respect to an os from an extended sample are the following two cases:
	$$\E(X_{1:2}|X_{2:4})=aX_{2:4}+b,
\quad \mbox{ and }\quad \E(X_{2:2}|X_{3:4})=aX_{3:4}+b.$$
	Actually each of these two conditions was written in DW in the expanded integral form. Unfortunately there are misprints in those formulas: ``y" is missing under all integrals and in the second equation the coefficients of two integrals should be: 1/3 instead of 1/6 for the first integral and 1/2 instead of 1/3 for the second.
	
	We are able to solve these problems only when $a=1$ and $b=0$.
	\begin{corollary}
		$\,$
	\begin{enumerate}
		\item If $\E|X_{1:2}|<\infty$ and $\E(X_{1:2}|X_{2:4})=X_{2:4}$,
 then $q(u)\propto \tfrac{1+u}{u^{5/4}(1-u)^{5/2}}$, $u\in(0,1)$.
		\item If $\E|X_{2:2}|<\infty$ and $\E(X_{2:2}|X_{3:4})=X_{3:4}$, then $q(u)\propto \tfrac{2-u}{u^{5/2}(1-u)^{5/4}}$, $u\in(0,1)$.
	\end{enumerate}
\end{corollary}
\begin{proof}
	These results follow directly from Th. \ref{2.2} by taking: in the first case $j=2$ and $n=4$ and thus $\lambda=1/4$; in the second case $j=3$ and $n=4$ and thus $\lambda=3/4$.
	\end{proof}
\vspace{5mm}

	From the proof of Th. \ref{2.2} it follows that if $\E|X_{j-1:n}|<\infty$ and $\E|X_{j+1:n}|<\infty$ and \eqref{akuk} holds for an arbitrary (but fixed) $\lambda\in(0,1)$ then the quantile density of $X_1$ has the form given in \eqref{qdf}. This is a direct extension of Th. 3.1 of ABN (and Th. 2 of Nevzorov (2002)) which follows by taking $n=2i+1$ and $j=i+1$. Note that this is the only case among possible forms of $q$ in \eqref{qdf} when the distribution of $X_1$ is of the complementary beta form.

%
%
%
%
%
%
%

\subsection{OS from the extended sample given os from the original sample}\label{nm}
In this subsection we still keep the assumption that $m<n$ but the conditioning now will be with respect to $X_{i:m}$.

From (\ref{140410-1856}), we derive the conditional density of $X_{j:n}|X_{i:m}=x$
with respect to $\nu _{x}$ as
\begin{eqnarray*}
f_{X_{j:n}|X_{i:m}=x}(y) &=&\sum_{k=i}^{i+n-m}\,\tfrac{\binom{k-1}{i-1}%
\binom{n-k}{m-i}}{\binom{n}{m}}\tfrac{f_{k:n}(x)}{f_{i:m}(x)}%
f_{X_{j:n}|X_{k:n}=x}(y) \\
&=&\sum_{k=i}^{i+n-m}\,\binom{n-m}{k-i}\,F^{k-i}(x)\bar{F}%
^{n-m-(k-i)}(x)f_{X_{j:n}|X_{k:n}=x}(y),
\end{eqnarray*}
and consequently we have the representation
\begin{equation}
\mathbb{E}(X_{j:n}|X_{i:m}=x)=\sum_{\ell =0}^{n-m}\,\binom{n-m}{\ell }
\,F^{\ell }(x)\bar{F}^{n-m-\ell }(x)\,\mathbb{E}(X_{j:n}|X_{i+\ell :n}=x).
\label{expe2}
\end{equation}
It is known that if $j>i+\ell $, then the conditional distribution of $\P_{X_{j:n}|X_{i+\ell :n}=x}$ is the same as the distribution of
the $(j-i-\ell )$th os obtained from an i.i.d. sample of size $
(n-i-\ell )$ from a parent distribution function $\frac{
F(t)-F(x)}{1-F(x)},$ $x<t<\infty $; if $j<i+\ell $, it is the same as the
distribution of the $j$th os in a sample of size $i+\ell -1$
from the parent distribution function $\frac{F(t)}{F(x)},$ $
-\infty <t<x$. Using these facts, after some algebraic manipulation, $
\mathbb{E}(X_{j:n}|X_{i:m}=x)$ can be expressed in a more explicit form. In general, characterizations (or identifiability question) through the form
of $\mathbb{E}(X_{j:n}|X_{i:m})$  seem to be difficult and in some cases linearity of such conditional expectation is plainly non-admissible. Consequently, while discussing characterizations we will restrict our considerations only to several tractable cases.

In the following lemma we derive relatively simple representations of $\mathbb{E}(X_{j:n}|X_{i:m})$ in special cases which will be used in characterizations later on in this subsection. For these special cases we  provide straightforward  proofs which is an alternative to derivations based on the general formula \eqref{expe2}.

\begin{lemma}
\label{lem.1}For $m<n,$ $1\leq i\leq m$ and $1\leq j\leq n$, we have

\begin{enumerate}
\item[(i)] $\mathbb{E}(X_{1:n}|X_{1:m}=x)=x\bar{F}^{n-m}(x)+(n-m)\int_{-%
\infty }^{x}t\bar{F}^{n-m-1}(t)f(t)dt,$

\item[(ii)] $\mathbb{E}(X_{n:n}|X_{m:m}=x)=xF^{n-m}(x)+(n-m)\int_{x}^{\infty
}tF^{n-m-1}(t)f(t)dt,$

\item[(iii)] $\mathbb{E}(X_{i:m+1}|X_{i:m}=x)=x\bar{F}(x)+\frac{i}{F^{i-1}(x)%
}\int_{-\infty }^{x}tF^{i-1}(t)f(t)dt,$

\item[(iv)] $\mathbb{E}(X_{j:n}|X_{1:1}=x)=\int_{-\infty
}^{x}tf_{j:n-1}(t)dt+x\binom{n-1}{j-1}F^{j-1}(x)\bar{F}%
^{n-j}(x)+\int_{x}^{\infty }tf_{j-1:n-1}(t)dt.$
\end{enumerate}
\end{lemma}

\begin{proof}
First we have
\begin{eqnarray*}
\mathbb{E}(X_{1:n}|X_{1:m}=x) &=&x\text{ }\mathbb{P}(\min \{X_{m+1},\cdots
,X_{n}\}>x|X_{1:m}=x) \\
&&+\mathbb{E}(\min \{X_{m+1},\cdots ,X_{n}\}I_{\{\min \{X_{m+1},\ldots
,X_{n}\}<x\}}|X_{1:m}=x)
\end{eqnarray*}
and
\begin{eqnarray*}
\mathbb{E}(X_{n:n}|X_{m:m}=x) &=&x\text{ }\mathbb{P}(\max \{X_{m+1},\ldots
,X_{n}\}<x|X_{m:m}=x) \\
&&+\mathbb{E}(\max \{X_{m+1},\ldots ,X_{n}\}I_{\{\max \{X_{m+1}\ldots
,X_{n}\}>x\}}|X_{m:m}=x)
\end{eqnarray*}
and the assertions (i) and  (ii) follow immediately.

Next, note that
\begin{equation*}
X_{i:m+1}=X_{i:m}I_{\{X_{m+1}>X_{i:m}\}}+X_{m+1}I_{
\{X_{i-1:m}<X_{m+1}<X_{i:m}\}}+X_{i-1:m}I_{\{X_{m+1}<X_{i-1:m}\}}.
\end{equation*}
Consequently,
\begin{eqnarray*}
\mathbb{E}(X_{i:m+1}|X_{i:m}=x) &=&x\bar{F}(x)+\int_{-\infty }^{x}t\mathbb{P}
(X_{i-1:m}<t|X_{i:m}=x)\,f(t)dt \\
&&+\mathbb{E}(X_{i-1:m}F(X_{i-1:m})|X_{i:m}=x) \\
&=&x\bar{F}(x)+\int_{-\infty }^{x}t\,\left( \tfrac{F(t)}{F(x)}\right)
^{i-1}\,f(t)dt \\
&&+\int_{-\infty }^{x}yF(y)\,(i-1)\tfrac{F^{i-2}(y)}{F^{i-1}(x)}\,f(y)\,dy
\end{eqnarray*}
and this yields the assertion (iii).

Finally, the assertion (iv) follows from%
\begin{equation*}
\mathbb{E}(X_{j:n}|X_{1:1}=x)=\mathbb{E}(X_{j:n}|X_{n}=x)
\end{equation*}
and
\begin{equation*}
X_{j:n}=X_{j:n-1}I_{\{X_{n}>X_{j:n-1}\}}+X_{n}I_{
\{X_{j-1:n-1}<X_{n}<X_{j:n-1}\}}+X_{j-1:n-1}I_{\{X_{n}<X_{j-1:n-1}\}}.
\end{equation*}
The proof is completed.
\end{proof}
Our main objective in this subsection is to show that the shape of the regresion curves studied in Lem. \ref{lem.1} determine the parent distribution $F$.

In the remaining part of this subsection we assume that the density $f=F'$ is strictly positive on $(a, b):=\{x\in \mathbb{R}: \, 0<F(x)<1\}$ with $-\infty\le a<b\le \infty$.

First we will consider the cases (i) and (ii) of Lem. \ref{lem.1}. It appears that under mild conditions  $\mathbb{E}(X_{1:n}|X_{1:m})$ and $\mathbb{E}
(X_{n:n}|X_{m:m})$ determine the parent cdf $F$.



\begin{theorem}
\label{181008-2005}
Suppose that $g$ is a differentiable function on $(a,b)$.
\begin{enumerate}
\item[(A)] If $\E|X_{1:n}|<\infty$ and
\begin{equation}
\mathbb{E}(X_{1:n}|X_{1:m})=g(X_{1:m}), \label{eq01}
\end{equation}
then
	\begin{enumerate}
	\item[(A.1)] $\lim_{x\to a+}g(x)=a$;
	\item[(A.2)] $g'$ is a decreasing function with  $\lim_{x\to a^+}g'(x)=1$ and  $\lim_{x\to b^-}g'(x)=0$;
	\item[(A.3)]$F(x)=1-\sqrt[n-m]{g^{\prime }(x)}$, for $x\in (a,b)$.
	\end{enumerate}
\item[(B)] If $\E|X_{n:n}|<\infty$ and
\begin{equation}
\mathbb{E}(X_{n:n}|X_{m:m})=g(X_{m:m}),
\label{181009-1808}
\end{equation}
then
	\begin{enumerate}
	\item[(B.1)] $\lim_{x\to b^-}g(x)=b$;
	\item[(B.2)] $g'$ is an increasing function with $\lim_{x\to a^+}g'(x)=0$ and $\lim_{x\to b^-}g'(x)=1$;
	\item[(B.3)]$F(x)=\sqrt[n-m]{g^{\prime }(x)}$, for $x\in (a,b)$.
	\end{enumerate}	
\end{enumerate}
\end{theorem}
\begin{proof}
	Since there is an obvious duality between the two cases (it suffices to consider negative of the original observations to move between (A) and (B)) we provide only the proof of (A).
	
From (i) of Lem. \ref{lem.1} and \eqref{eq01} we obtain the equation
\begin{equation}
x\bar{F}^{n-m}(x)+(n-m)\int_{a }^{x}t\bar{F}^{n-m-1}(t)f(t)dt=g(x),
\,\,\, x\in (a,b),
\label{eq05}
\end{equation}
then (A.1) follows easily by taking limits $x\to a^+$ in both sides of (\ref{eq05}).
Differentiating (\ref{eq05}), after elementary algebra,  we get
\begin{equation*}
\bar{F}^{n-m}(x)
=g^{\prime }(x),\,\,\, x\in (a,b),
\end{equation*}
then (A.2) follows from the well-known properties of a cdf and (A.3) is immediate.
\end{proof}

Note that  linearity of regression in \eqref{eq01} (or \eqref{181009-1808}) is impossible since it would lead to $F$ being constant.

\vspace{5mm}
As an illustration of Th. \ref{181008-2005} we provide two examples:
\begin{itemize}
	\item  if either
\begin{equation*}
\mathbb{E}(X_{1:n}|X_{1:m}=x)=\frac{1-(1-x)^{n-m+1}}{n-m+1},\text{ }0\le x \le 1,
\end{equation*}
or
\begin{equation*}
\mathbb{E}(X_{n:n}|X_{m:m}=x)=\frac{x^{n-m+1}+n-m}{n-m+1},\text{ }0\le x \le 1.
\end{equation*}
then the parent distribution is uniform on $(0,1)$;
\item if
$$
\E (X_{1:n}\mid X_{1:m}=x)=\tfrac{1-\exp(-(n-m)x)}{n-m}, \;\;\; x\ge 0,
$$
then the parent distribution is standard (i.e. with mean 1) exponential.
\end{itemize}

\vspace{5mm}
Now we will consider case (iii) of Lem. \ref{lem.1}.

\begin{theorem}
\label{thh}
Suppose $\E |X_{i:m+1}|<\infty$ for certain integers $2\le i\le m$ and let $h$ be a differentiable function on $(a,b)$. If
%
%
\begin{equation*}
\mathbb{E}(X_{i:m+1}|X_{i:m})=X_{i:m}-h(X_{i:m}).
\end{equation*}
Then
\begin{equation}
F(x)=
\displaystyle
\frac{h^{-\frac{1}{i-1}}(x)}{h^{-\frac{1}{i-1}}(b^-)
+\displaystyle\frac{1}{i-1}
\int_x^b  h^{-\frac{i}{i-1}}(t)\,dt},\qquad x\in
(a,\,b). \label{Fx}
\end{equation}
\end{theorem}

\begin{proof}
Using  Lem. \ref{lem.1}(iii)  we have
$$
x\overline{F}(x)
+
\frac{i}{F^{i-1}(x)}
\int_a^x t F^{i-1}(t) f(t)\, dt
=x-h(x),\,\,x\in (a,b)
$$
and, after simple algebra,
$$
F^{i-1}(x)h(x)=xF^{i}(x)-i\int_{a }^{x}\,tF^{i-1}(t)f(t)\,dt, \,\,x\in (a,b).
\label{eq06}
$$
Integrating by parts (observe that if $a=-\infty$ then $\lim_{x\to a^+} xF^i(x)=0$ due to the hypothesis $\E|X_{i:m+1}|<\infty$ and Lem. \ref{iint}),
\begin{equation}
\label{181009-2019}
F^{i-1}(x) h(x)=\int_a^x F^i(t)\,dt
 ,\,\,x\in (a,b),
\end{equation}
from which we conclude that $h(x)>0$ for $x\in (a,b)$.

Let $G(x)=\int_{a}^x F^{i}(t)\,dt$, $x\in (a,b]$, so
\begin{equation}
\label{181011-1254}
F(x)=G'(x)^{\frac{1}{i}},\qquad x\in (a,b),
\end{equation}
and, after some algebra, \eqref{181009-2019} yields
\begin{equation*}
G^{-\frac{i}{i-1}}(x)\, G'(x)= h^{-\frac{i}{i-1}}(x),\qquad x\in (a,b).
\end{equation*}
Therefore,
\begin{equation*}
\int_x^b G^{-\frac{i}{i-1}}(t)\, G'(t)\,dt
=
\int_x^b  h^{-\frac{i}{i-1}}(t)\,dt
,\qquad x\in(a,b)
\end{equation*}
or equivalently
\begin{equation*}
G^{-\frac{1}{i-1}}(x)-G^{-\frac{1}{i-1}}(b)=\frac{1}{i-1}\int_x^b  h^{-\frac{i}{i-1}}(t)\,dt
,\qquad x\in(a,b).
\end{equation*}
From \eqref{181009-2019} it follows that $h(b^-)=G(b)$, and thus
\begin{equation}
G(x)=
\left(
h^{-\frac{1}{i-1}}(b^-)
+\displaystyle\frac{1}{i-1}
\int_x^b  h^{-\frac{i}{i-1}}(t)\,dt
\right)^{-(i-1)}
,\,\,\,x\in(a,b)
\end{equation}
(if $h(b^-)=\infty$, which is possible, we take  $h(b^-)^{-\frac{1}{i-1}}=0$).
The result follows now easily by \eqref{181011-1254}.
\end{proof}

As an illustration of Th. \ref{thh} we provide some examples:
\begin{itemize}
	\item if for $\alpha>0$
	\begin{equation*}
	\E\left(X_{i:m+1}\mid X_{i:m}=x\right) =x-\tfrac{x^{\alpha+1}}{i\alpha+1},\qquad 0
	\le x \le 1,
	\end{equation*}
	then the parent distribution is power with the cdf $F(x)=x^\alpha$, $x\in [0,1]$.
	\item if for $A>0$  and $r>0$ such that $ri>1$
	 \begin{equation*}
	 \E\left( X_{i:m+1}\mid X_{i:m}=x\right) =x-\tfrac{(1+A(b-x))^{-r+1}}{A(ri-1)},\qquad x\le b,
	 \end{equation*}
	 then the parent distribution is (negative) Type IV Pareto distribution with the cdf  $F(x)=(1+A(b-x))^{-r}$, $x\in(-\infty, b]$;
	
	 \item if we specialize the above example by fixiing $r=1$ we obtain characterization of cdf $F(x)=\tfrac{1}{1+A(b-x)}$, $x\le b$, by the linearity of regression
	 $$
	 \E(X_{i:m+1}|X_{i:m}=x)=x-\frac{1}{A(i-1)},\qquad x\le b;
	 $$
	 \item if for $\lambda>0$
	 $$
	 \E\left(X_{i:m+1}\mid X_{i:m}=x\right) =x-\tfrac{\exp(\lambda x)}{i\lambda},\qquad x\le 0.
	 $$
	 then the parent distribution is negative exponential with the cdf $F(x)=\exp(\lambda x)$, $x\in (-\infty, 0]$.
\end{itemize}

\vspace{5mm}
Characterization by $\E(X_{1:1}|X_{j:n})$ (of course, $X_{1:1}=X_1$) was studied in WG and Balakrishnan and Akhundov (2003) in the linear case. In case (iv) of Lem. \ref{lem.1} we will consider the dual conditional expectation $\E(X_{j:n}|X_{1:1})$. We are not able to express the parent distribution in terms of this regression function in this case. Instead we solve a more modest question of identifiability of the distribution of $X_1$.

\begin{theorem}
\label{thm3.5}
Suppose $\E |X_{j:n}|<\infty$ for certain integers $1\le j\le n$. Then the conditional
expectation $\mathbb{E}(X_{j:n}|X_{1:1})$ uniquely determines the parent cdf $F$.
\end{theorem}

\begin{proof}
Denote $h(x)=\mathbb{E}(X_{j:n}|X_{1:1}=x)$, $x\in (a,\,b)$. From (iv)
of Lem. \ref{lem.1},
\begin{equation*}
h(x)=\int_{-\infty }^{x}tf_{j:n-1}(t)dt+x\binom{n-1}{j-1}F^{j-1}(x)
\bar{F}^{n-j}(x)+\int_{x}^{\infty }tf_{j-1:n-1}(t)dt,\,\,\, x\in (a,b).
\end{equation*}
Then differentiating the above equation with respect to $x$ we get
\begin{equation}\label{hprim}
h^{\prime }(x)=\binom{n-1}{j-1}F^{j-1}(x)\bar{F}^{n-j}(x),\,\,\, x\in (a,b).
\end{equation}

Let us assume that the $H$ is not unique, that is there exist two
different distribution functions $F$ and $G$ with the same support such that
$H$ is the same for $F$ and $G$. Hence
\begin{equation*}
F^{\tfrac{j-1}{n-j}}(x)\bar{F}(x)=G^{\tfrac{j-1}{n-j}}(x)\bar{G}(x),\,\,\, x\in (a,b),
\end{equation*}
which can be rewritten as
\begin{equation}
\int_{G(x)}^{F(x)}\left( (j-1)t^{\tfrac{j-1}{n-j}-1}-(n-1)t^{\tfrac{j-1}{n-j}}\right) dt=0,\,\,\, x\in (a,b).  \label{inte}
\end{equation}
Note that, for $0<t<\tfrac{j-1}{n-1},$ the integrand in \eqref{inte} is
strictly positive. Therefore $F(x)=G(x)$ in a right neighbourhood of the
left end of the support. Consequently, we have $x_{0}=\sup \{x\geq
a:\,F(x)=G(x)\}>a$ and by continuity, $F(x_{0})=G(x_{0})$.

Let us prove that $F(x_{0})=G(x_{0})\geq \tfrac{j-1}{n-1}$. Assume the
opposite, $F(x_{0})=G(x_{0})<\tfrac{j-1}{n-1}$. Then, by continuity of $F$
and $G,$ there exists $\epsilon >0$ such that $F(x_{0}+\epsilon )<\tfrac{j-1
}{n-1}$ and $G(x_{0}+\epsilon )<\tfrac{j-1}{n-1}$. Hence again the integrand
in \eqref{inte} is strictly positive and we get $F(x_{0}+\epsilon
)=G(x_{0}+\epsilon )$ which contradicts the definition of $x_{0}$.
Therefore, $F(x_{0})=G(x_{0})\geq \tfrac{j-1}{n-1}$. Consider now an
arbitrary $x>x_{0}$. Since $F$ and $G$ are strictly increasing on $(a,b)$ we
see that $F(x)>\tfrac{j-1}{n-1}$ and $G(x)>\tfrac{j-1}{n-1}$. But for $t>
\tfrac{j-1}{n-1}$ the integrand in \eqref{inte} is strictly negative.
Consequently $F(x)=G(x)$.
\end{proof}

Note that due to \eqref{hprim} the derivative of the regression function can be useful for determining the parent cdf. In particular, it follows from \eqref{hprim} that
(a) if $\E(X_{1:n}|X_{1:1}=x)=g(x)$ is differentiable then $F(x)=1-\sqrt[n-1]{g'(x)}$;
(b) if $\E(X_{n:n}|X_{1:1}=x)=g(x)$ is differentiable then $F(x)=\sqrt[n-1]{g'(x)}$. Actually, these results  are also covered by Th. \ref{181008-2005} for $m=1$.


Finally, we use \eqref{hprim} to derive two new characterizations of the logistic distribution.
\begin{corollary}
Assume that either
$$
\E(X_{2:3}|X_{1:1}=x)=\tfrac{2e^x}{1+e^x},\quad x\in\R,
$$
or with unknown parent cdf $F$
$$
\E(X_{2:3}|X_{1:1}=x)\propto F(x),\quad x\in\R,
$$
or
$$
\E(X_{2:3}|X_{1:1}=x)\propto \bar{F}(x),\quad x\in\R.
$$

\noindent Then $X_1$, $X_2$ and $X_3$ have the logistic distribution.
\end{corollary}
\begin{proof}
In the first case \eqref{hprim} implies
$$
F(x)\bar{F}(x)=\tfrac{e^x}{(1+e^x)^2}.
$$
There are two solutions of the above quadratic equation in the unknown $F(x)$. Only one of them, $F(x)=\tfrac{e^x}{1+e^x}$, $x\in\R$, gives the valid (logistic) distribution function.

In the remaining two cases \eqref{hprim} yields $f\propto F\bar{F}$, i.e. we obtain a distribution from $CB(1,1)$ family, which is the family of logistic laws - see e.g. Galambos (1991).
\end{proof}
The last two cases in the above corollary can be easily generalized:
\begin{itemize}
\item if $\E(X_{j:n}|X_{1:1}=x)\propto F^s(x)$, $s\in\R$, then the parent distribution belongs to $CB(j-s,n-j)$;
\item if $\E(X_{j:n}|X_{1:1}=x)\propto \bar{F}^s(x)$, $s\in\R$, then the parent distribution belongs to $CB(j-1,n-j-s+1)$.
\end{itemize}


\section{Conclusion}
\label{sec.6}
The aim of this paper is two-fold: (1) derivation of bivariate distribution of os's from overlapping samples in the general overlapping scheme; (2) investigations of regression properties of os's from overlapping samples, in particular, extension of characterizations by linearity of regression of os's or identifiability results to the overlapping situation. Throughout the paper we assumed that the original observations are iid and their common distribution is absolutely continuous with respect to the Lebesque measure. The first task was fully resolved. Though the general formula is quite complicated, in several important special cases it gives quite transparent formulas and can be useful e.g.  in studying moving order statistics or analyzing conditional structure of os's from overlapping samples. Regarding the second task we identified new settings in which  linearity of regression or the general form of the regression function characterizes  the parent  distributions, in several other cases  uniqueness results were obtained instead. However, the issue of characterizing of the parent cdf $F$  by using a general relation
$$\E\left(X_{i:m}|X_{j:n}^{(r)}\right)=h(X_{j:n}^{(r)}),$$
where  $h:\mathbb{R}\rightarrow \mathbb{R}$ is a Borel function, remains open and seems to be rather difficult to settle.

Finally let us mention that in the special case of $\E(X_{i:m}|X_{j:n})$, due to  \eqref{expe1}, the problem we studied embeds naturally in the question of characterization of the parent distribution by regression of $L$-statistics of the form $\sum_{i=1}^n a_iX_{i:n}$ on a single os $X_{j:n}$.

\vspace{0.8cm}
\noindent\textbf{Acknowledgments.}
 The authors are greatly indebted to the referees for valuable suggestions which
helped significantly in improving the final version of the paper. Support for this research was provided in part by the Ministry of
Science and Technology, Taiwan, Grant No. 102-2118-M-305-003 and
104-2118-M-305-003, by Ministerio de Economia, Industria y Competitividad, Spain, Grant MTM2016-74983 and by the National
Science Center, Poland, Grant 2016/21/B/ST1/00005.

\vspace*{0.8cm}
\section*{\bf Appendix}
\label{sec.5}

\begin{proof}[Proof of Prop. \ref{prawd}]
\textbf{(i) }$\boldsymbol{k<l}$\textbf{.} Note that
\begin{equation}
\label{k<l}
\{X_{i:m}=X_{k:n+r},\,X_{j:n}^{(r)}=X_{l:n+r}\}=\bigcup_{\substack{\alpha\in A\cup B \\ \beta\in B\cup C \\ \alpha\neq \beta}}\,\{X_{i:m}=X_{k:n+r}=X_{\alpha },\,X_{j:n}^{(r)}=X_{l:n+r}=X_{\beta
}\}
\end{equation}
and the sets under $\bigcup$ at the right-hand side are pair-wise disjoint. Moreover, for any distinct $\alpha\in A\cup B$ and $\beta\in B\cup C$
\begin{equation}\label{union1}
\{X_{i:m}=X_{k:n+r}=X_{\alpha },\,X_{j:n}^{(r)}=X_{l:n+r}=X_{\beta
}\}=\bigcup_{\sigma \in S(\alpha ,\beta )}\,\{X_{\sigma (1)}\leq \ldots \leq
X_{\sigma (r+n)}\},
\end{equation}
where
\begin{eqnarray*}
S(\alpha ,\beta ) &=&\left\{ \sigma \in \mathcal{S}_{n+r}:\,\sigma
(k)=\alpha ,\,\sigma (l)=\beta ,\,|\sigma (\{1,\ldots ,k-1\})\cap (A_{\alpha
}\cup B_{\alpha ,\beta })|=i-1,\right. \\
&&\hspace{0.2cm}\left. \left\vert \sigma (\{1,\ldots ,l-1\}\setminus
\{k\})\cap (B_{\alpha ,\beta }\cup C_{\beta })\right\vert =j-1-I_{B}(\alpha
)\right\}.
\end{eqnarray*}
Here and in the sequel we denote $U_{x_{1},\ldots ,x_{K}}:=U\setminus \{x_{1},\ldots ,x_{K}\}$ for any set
$U$. Since the sets $A_{\alpha }$, $B_{\alpha ,\beta }$ and $C_{\beta }$ are
disjoint and
\begin{equation*}
A_{\alpha }\cup B_{\alpha ,\beta }\cup C_{\beta }=(A\cup B\cup C)_{\alpha
,\beta }
\end{equation*}
it follows that
\begin{eqnarray*}
S(\alpha ,\beta ) &=&\left\{ \sigma \in \mathcal{S}_{n+r}:\,\sigma
(k)=\alpha ,\,\sigma (l)=\beta ,\text{ }|\sigma (\{1,\ldots ,k-1\})\cap
C_{\beta }|=k-i,\right. \\
&&\hspace{0.2cm}\left. |\sigma (\{1,\ldots ,l-1\}\setminus \{k\})\cap
A_{\alpha }|=l-j-1+I_{B}(\alpha )\right\} .
\end{eqnarray*}
Therefore, by Lem. \ref{DD}
\begin{equation*}
|S(\alpha ,\beta )|=\mathfrak{D}_{|A|-I_{A}(\alpha
),\,|B|-I_{B}(\alpha )-I_{B}(\beta ),\,|C|-I_{C}(\beta
),\,k-1,\,l-k-1,\,k-i,\,l-j-I_{A}(\alpha )} .
\end{equation*}
Note that the sets under the $\bigcup $ sign in \eqref{union1} are pair-wise disjoint $\mathbb{
P}$-a.s. and each of them has probability $1/(n+r)!$ Therefore
\begin{equation*}
P(\alpha ,\beta ):=\mathbb{P}\left( \{X_{i:m}=X_{k:n+r}=X_{\alpha
},\,X_{j:n}^{(r)}=X_{l:n+r}=X_{\beta }\}\right) =\tfrac{|S(\alpha ,\beta )|}{
(n+r)!}.
\end{equation*}

There are four possible cases for the triplet $(A_{\alpha },B_{\alpha ,\beta
},C_{\beta })$:
\begin{equation*}
(A_{\alpha },B_{\alpha ,\beta },C_{\beta })=\left\{
\begin{array}{ll}
(A_{\alpha },B_{\beta },C), & \mbox{if}\;\alpha \in A,\;\beta \in B, \\
(A_{\alpha },B,C_{\beta }), & \mbox{if}\;\alpha \in A,\;\beta \in C, \\
(A,B_{\alpha ,\beta },C), & \mbox{if}\;\alpha ,\beta \in B, \\
(A,B_{\alpha },C_{\beta }), & \mbox{if}\;\alpha \in B,\;\beta \in C.
\end{array}%
\right.
\end{equation*}
That is, following \eqref{k<l} we obtain
\begin{eqnarray*}
&&\mathbb{P}\left( \{X_{i:m}=X_{k:n+r},\,X_{j:n}^{(r)}=X_{l:n+r}\right) \\
&=&\sum_{\substack{ \alpha \in A,  \\ \beta \in B}}\;P(\alpha ,\beta )+\sum
_{\substack{ \alpha \in A,  \\ \beta \in C}}\;P(\alpha ,\beta )+\sum_{\alpha
,\,\beta \in B}\;P(\alpha ,\beta )+\sum_{\substack{ \alpha \in B,  \\ \beta
\in C}}\;P(\alpha ,\beta ) \\
&=&\tfrac{|A||B||S(r,m)|+|A||C||S(r,n+r)|+|B|(|B|-1)|S(m-1,m)|+|B||C|
|S(m,n+r)|}{(r+n)!} \\
&=&\tfrac{|A||B|\mathfrak{D}_{|A|-1,\,|B|-1,\,|C|,\,k-1,\,l-k-1,\,k-i,\,l-j-1}
}{(r+n)!}+\tfrac{|A||C|\mathfrak{D}_{|A|-1,\,|B|,\,|C|-1,\,k-1,\,l-k-1,\,k-i,
\,l-j-1}}{(r+n)!} \\
&&+\tfrac{|B|(|B|-1)\mathfrak{D}_{|A|,\,|B|-2,\,|C|,\,k-1,\,l-k-1,\,k-i,\,l-j}
}{(r+n)!}+\tfrac{|B||C|\mathfrak{D}_{|A|,\,|B|-1,\,|C|-1,\,k-1,\,l-k-1,\,k-i,
\,l-j}}{(r+n)!}.
\end{eqnarray*}
Denote numerators in subsequent four fractions above by $I_{1}$, $I_{2}$, $
I_{3}$ and $I_{4}$, respectively. Note that
$$
I_{1}+I_{2}=|A|[|B|\mathfrak{D}_{|A|-1,\,|B|-1,\,|C|,\,k-1,\,l-k-1,\,k-i,
\,l-j-1} +|C|\mathfrak{D}_{|A|-1,\,|B|,\,|C|-1,\,k-1,\,l-k-1,\,k-i,\,l-j-1}] $$
$$=|A|\tfrac{(|A|+|B|+|C|-2)!}{\binom{|A|+|B|+|C|-2}{k-1,l-k-1}}\binom{|A|-1
}{l-j-1}\,\left[ \binom{|C|}{k-i}\sum_{m=1}^{l-j-1}\,\binom{l-j-1}{m}\,|B|
\binom{|B|-1}{i-m-1}\,\binom{|B|+|C|+m-k}{j+m-k}\right. $$
$$\left. +|C|\binom{|C|-1}{k-i}\sum_{m=1}^{l-j-1}\,\binom{
l-j-1}{m}\,\binom{|B|}{i-m-1}\,\binom{|B|+|C|+m-k}{j+m-k}\right] .
$$
We will use several times the following elementary identity
\begin{equation}
s\binom{s-1}{r}=(s-r)\binom{s}{r}.  \label{id}
\end{equation}

Applying \eqref{id} at the right hand side above we get
\begin{eqnarray*}
I_{1}+I_{2} &=&|A|\tfrac{(|A|+|B|+|C|-2)!}{\binom{|A|+|B|+|C|-2}{k-1,l-k-1}}
\binom{|A|-1}{l-j-1}\binom{|C|}{k-i} \\
&&\times \sum_{m=0}^{l-j-1}\,\binom{l-j-1}{m}\binom{|B|}{i-m-1}\,\binom{
|B|+|C|+m-k}{j+m-k}(|B|+1+m+|C|-k) \\
&=&|A|(|B|+|C|-j+1)\,\tfrac{(|A|+|B|+|C|-2)!}{\binom{|A|+|B|+|C|-2}{k-1,l-k-1
}}\binom{|A|-1}{l-j-1}\binom{|C|}{k-i} \\
&&\times \,\sum_{m=0}^{l-j-1}\,\binom{l-j-1}{m}\binom{|B|}{i-m-1}\binom{
|B|+|C|+m-k+1}{j+m-k} \\
&=&\tfrac{|A|(|B|+|C|-j+1)}{|A|+|B|+|C|-l+1}\,\mathfrak{D}_{|A|-1,|B|,|C|,k-1,l-k-1,k-i,l-j-1}.
\end{eqnarray*}
Similarly to $I_{1}+I_{2}$, we can also obtain a explicit form of $
I_{3}+I_{4} $. Then combining the expressions for $I_{1}+I_{2}$ and $
I_{3}+I_{4}$ we get the final formula in this case.

\textbf{(ii) }$\boldsymbol{k=l}$\textbf{.} Then
\begin{equation}
\label{k=l}
\{X_{i:m}=X_{k:n+r}=X_{j:n}^{(r)}\}=\bigcup_{\alpha\in B}\,\{X_{i:m}=X_{k:n+r}=X_{j:n}^{(r)}=X_{\alpha }\}=\bigcup_{\alpha\in B}\,\bigcup_{\sigma\in S(\alpha )}\,\{X_{\sigma(1)}\le \ldots\le X_{\sigma(n+r)}\},
\end{equation}
where
\begin{eqnarray*}
S(\alpha ) &=&\{\sigma \in \mathcal{S}(n+r):\,\sigma (k)=\alpha ,\,|\sigma
(\{1,\ldots ,k-1\})\cap (A\cup B_{\alpha })|=i-1, \\
&&\hspace{0.2cm}\,|\sigma (\{1,\ldots ,k-1\})\cap (B_{\alpha }\cup C)|=j-1\}
\\
&=&\{\sigma \in \mathcal{S}(n+r):\,\sigma (k)=\alpha ,\,|\sigma (\{1,\ldots
,k-1\})\cap C|=k-i, \\
&&\hspace{0.2cm}\,|\sigma (\{1,\ldots ,k-1\})\cap A|=k-j\}.
\end{eqnarray*}
By Lem. \ref{DD} it follows that
\begin{equation*}
|S(\alpha )|=\mathfrak{D}_{|A|,|B|-1,|C|,k-1,0,k-i,k-j}.
\end{equation*}
Since the right-hand side of \eqref{k=l} is the union of pair-wise disjoint sets having the same probability $1/(n+r)!$ we get immediately the final formula in this case.

\textbf{(iii) }$\boldsymbol{k>l}$\textbf{.} Note that in this case  \eqref{k<l} and \eqref{union1} remain formally valid however this time the set
 $S(\alpha, \beta)$ is different:
\begin{eqnarray*}
S(\alpha ,\beta ) &=&\{\sigma \in \mathcal{S}_{n+r}:\,\sigma (k)=\alpha
,\,\sigma (l)=\beta ,\,|\sigma (\{1,\ldots ,l-1\})\cap (B_{\alpha ,\beta
}\cup C_{\beta })|=j-1, \\
&&\hspace{0.2cm}\left\vert \sigma (\{1,\ldots ,k-1\}\setminus \{l\})\cap
(A_{\alpha }\cup B_{\alpha ,\beta })\right\vert =i-1-I_{B}(\beta )\}.
\end{eqnarray*}
Consequently,
\begin{eqnarray*}
S(\alpha ,\beta ) &=&\{\sigma \in \mathcal{S}_{n+r}:\,\sigma (k)=\alpha
,\,\sigma (l)=\beta ,\,|\sigma (\{1,\ldots ,l-1\})\cap A_{\alpha }|=l-j, \\
&&\hspace{0.2cm}\left\vert \sigma (\{1,\ldots ,k-1\}\setminus \{l\})\cap
C_{\beta })\right\vert =k-i-1+I_{B}(\beta )\}.
\end{eqnarray*}
Therefore, according to Lem. \ref{DD}
\begin{equation*}
|S(\alpha ,\beta )|=\mathfrak{D}_{|C_{\beta }|,\,|B_{\alpha ,\beta
}|,|A_{\alpha }|,l-1,k-l-1,l-j,k-i-I_{C}(\beta )}.
\end{equation*}
Thus, analogously as in \textbf{Case (i)} we obtain
\begin{eqnarray*}
&&\mathbb{P}\left(X_{i:m}=X_{k:n+r},\,X_{j:n}^{(r)}=X_{l:n+r}\right) \\
&=&|B||C|\tfrac{|S(m,n+r)|}{(r+n)!}+|A||C|\tfrac{|S(r,n+r)|}{(n+r)!}
+|B|(|B|-1)\tfrac{|S(m-1,m)|}{(n+r)!}+|A||B|\tfrac{|S(r,m)|}{(n+r)!} \\
&=&\tfrac{|B||C|\mathfrak{D}_{|C|-1,\,|B|-1,\,|A|,\,l-1,\,k-l-1,\,l -j,\,k-i-1}
}{(r+n)!}+\tfrac{|A||C|\mathfrak{D}_{|C|-1,\,|B|,\,|A|-1,\,j-1,\,k-l-1,\,l-j,
\,k-i-1}}{(r+n)!} \\
&&+\tfrac{|B|(|B|-1)\mathfrak{D}_{|C|,\,|B|-2,\,|A|,\,l-1,\,k-l-1,\,l-j,\,k-i}
}{(r+n)!}+\tfrac{|A||B|\mathfrak{D}_{|C|,\,|B|-1,\,|A|-1,\,l-1,\,k-l-1,\,l-j,
\,k-i}}{(r+n)!}.
\end{eqnarray*}
This formula is the analogue of the respective one from \textbf{Case (i)} with
the roles of $|A|$ vs. $|C|$, $k$ vs. $l$ and $i$ vs. $j$ being exchanged. The final
result follows again by combining first two and second two numerators above with the use of \eqref{id}, similarly as it was done in \textbf{Case (i)}.
\end{proof}

\vspace*{0.8cm}

\noindent\textbf{References}

\begin{enumerate}
	\item {\sc Ahsanullah, M., Nevzorov, V.B.}, Generalized spacings of order statistics from extended sample. {\em J. Statist. Plann. Infer.} {\bf 85} (2000), 75-83.

	\item \textsc{Akhundov, I.S., Balakrishnan, N., Nevzorov, V.B.}, New
	characterizations by properties of midrange and related statistics. \emph{Comm. Statist. Theory Meth.} \textbf{33(12)} (2004), 3133-3134.
	
	\item {\sc Akhundov, I., Nevzorov, V.}, From Student's $t_2$ distribution to Student's $t_3$ distribution through characterizations. {\em Comm. Statist. Sim. Comp.} {\bf 41(6)} (2012), 710-716.
	
    \item \textsc{Arnold, B. C., Balakrishnan, N., Nagaraja, H. N.}, {\em A First Course in Order Statistics}, Vol. 54, (2008), SIAM.
	
	\item \textsc{Balakrishnan, N., Akhundov, I.S.}, A characterization by
	linearity of the regression function based on order statistics. \emph{Statist. Probab. Lett.} \textbf{63} (2003), 435-440.
	
	\item {\sc Balakrishnan, N., Tan, T.}, A parametric test for trend based on moving order statistics, {\em J. Statist. Comp. Simul.} {\bf 86(4)} (2016), 641-655.
	
	\item \textsc{Bieniek, M., Maciag, K.} Uniqueness of characterization of absolutely continuous distributions by regressions of generalized order statistics. {\em Adv. Statist. Anal.} \textbf{102(3)}
(2018), 359-380.

	\item \textsc{David, H. A. and Nagaraja, H. N.} {\em Order statistics}, Wiley Series in Probability and Statistics, 3rd ed, (2003), John Wiley \& Sons.
	
	\item \textsc{Dembi\'nska, A., Weso\l owski, J.}, Linearity of regression
	for non-adjacent order statistics. \emph{Metrika} \textbf{48} (1998),
	215-222.
	
	\item \textsc{Do\l egowski, A., Weso\l owski, J.}, Linearity of regression
	for overlapping order statistics. \emph{Metrika} \textbf{78} (2015), 205-218.
	
	\item {\sc Ferguson, T.}, On a Rao-Shanbhag characterization of the exponential/geometric distribution, {\em Sankhya A} {\bf 64} (2002), 246-255.
	
	\item {\sc Galambos, J.} Characterizations. In: {\em Handbook of the Logistic Distribution} (N. Balakrishnan, ed.), Dekker, 1991, 169-188.
	
	\item {\sc Hu, C.Y., Lin, G.D.} Characterizations of the logistic and related distributions. {\em J. Math. Anal. Appl.} {\bf 463} (2018), 79-92.
	
	\item {\sc Inagaki, N.} The distribution of moving order statistics. In: {\em Recent Developments in Statistical Inference and Data Analysis} (K. Matsusita, ed.), North Holland, Amsterdam 1980, 137-142.
	
	\item {\sc Ishida, I., Kvedaras, V.} Modeling autoregressive processes with
	moving-quantiles-implied nonlinearity. {\em Econometrics} {\bf 3} (2015), 2-54.

    \item {\sc Jones, M. C.} The complementary beta distribution. {\em J. Statist. Plann. Infer.} \textbf{104(2)} (2002),329-337.

    \item {\sc Jones, M. C.} On a class of distributions defined by the relationship between their density and distribution functions.\emph{Comm. Statist. Theory Meth.} \textbf{36(10)} (2007), 1835-1843.
	
	\item \textsc{Kamps, U.}, A general recurrence relation for moments of order
	statistics in a class of probability distributions and characterizations.
	\emph{Metrika} \textbf{38} (1991), 215-225.
	
\item \textsc{L\'opez Bl\'aquez, F.;  Moreno Rebollo, J. L.}, A characterization of distributions based on linear regression
             of order statistics and record values.
             \emph{Sankhy\=a Ser. A} \textbf{59}(3) (1997), 311-323.


	\item \textsc{L\'{o}pez-Bl\'{a}zquez, F., Salamanca-Mi\~no, B.} Maximal correlation in a non-diagonal case. {\em J. Multivar. Anal.} {\bf 131} (2014), 265-278.
	
	\item {\sc Marudova, N.M., Nevzorov, V.B.}, A class of distributions that includes Student's $t_2$-distribution. {\em J. Math. Sci.} {\bf 159(3)} (2009), 312-316.
	
	\item {\sc Mohie El-Din, M.M., Amein, M.M., Hamedani, G.G.}, On order statistics for GS-distributions. {\em J. Statist. Theory Appl.} {\bf 11(3)} (2012), 237-264.
	
	\item {\sc Mui\~no, J.M., Voit, E.O., Sorribas, A.}, GS-distributions: A new family of distributions for continuous unimodal variables. {\em Comput. Statist. Data Anal.} {\bf 50} (2006), 2769-2798.
	
	\item \textsc{Nagaraja, H.N., Nevzorov, V.B.}, On characterizations based on
	records values and order statistics. \emph{J. Statist. Plann. Infer.}
	\textbf{63} (1997), 271-284.
	
	\item \textsc{Nevzorov, V.B.}, Ona property of Student's distribution with two degrees of freedom. {\em Zap. Nauch. Sem. POMI} {\bf 294} (2002), 148-157. (in Russian; English translation: \emph{J. Math. Sci.} \textbf{127(1)} (2005), 757-1762)
	
	\item {\sc Nevzorov, V.B.}, On some regression relations of connecting sample averages and order statistics. {\em Vest. Sankt Petersburg Gos. Univ.} {\bf 60(2)} (2015), 364-368. (in Russian)
	
	\item \textsc{Nevzorov, V.B., Balakrishnan, N., Ahsanullah, M.}, Simple
	characterizations of Student $\mathrm{t}_2$-distribution. \emph{JRSS D}
	\textbf{52(3)} (2003), 395-400.
	
	\item {\sc Nevzorova, L., Nevzorov, V.B., Akhundov, I.}, A simple characterization of Student's $t_2$ distribution. {\em Metron} {\bf 55} (2007), 53-57.
	
	
	\item {\sc Richards, F.J.}, A flexible growth function for empirical use. {\em J. Exp. Bot.} {\bf 10} (1959), 290-300.
	
	\item {\sc Siddiqui, M.M.}, Order statistics of a sample and of an extended sample. In: {\em Nonparametric Techniques in Statistical Inference} (M.L. Puri, ed.), Cambridge Univ. Press 1970, 417-423.
	
	\item {\sc Tryfos, P., Blackmore, R.}, Forecasting records. {\em JASA} {\bf 80} (1985), 46-50.
	
	\item \textsc{Weso\l owski, J., Gupta, A.K.}, Linearity of convex mean
	residual life time. \emph{J. Statist. Plan. Infer.} \textbf{99} (2001),
	183-191.
	
	\item \textsc{Yanev, G., Ahsanullah, M.}, Characterizations of Student's $
	\mathrm{t}$-distribution via regressions of order statistics. \emph{Statist.}
	\textbf{46(4)} (2012), 429-435.
\end{enumerate}

\end{document}